\documentclass{amsart}
\usepackage{amsmath,amsthm,amssymb,amscd}

\newcommand{\fc}{\mathfrak{c}}

\newcommand{\ga}{\alpha}
\newcommand{\gb}{\beta}
\renewcommand{\gg}{\gamma}
\newcommand{\gd}{\delta}
\newcommand{\gw}{\omega}

\newcommand{\gS}{\Sigma}
\newcommand{\gs}{\sigma}
\newcommand{\eps}{\varepsilon}

\newcommand{\proj}{\mathrm{proj}}

\newcommand{\pprec}{\sqsubseteq}
\newcommand{\trash}{\mathrm{trash}}

\newcommand{\liff}{\leftrightarrow}

\newcommand{\cof}{\mathtt{cof}}

\newcommand{\cantor}{2^\gw}
\newcommand{\baire}{\gw^\gw}
\newcommand{\bintree}{2^{<\gw}}
\newcommand{\gwtree}{\gw^{<\gw}}

\newcommand{\dom}{\mathrm{dom}}

\newcommand{\power}{\mathcal{P}}

\newtheorem{theorem}{Theorem}[section]

\newtheorem{claim}[theorem]{Claim}
\newtheorem{corollary}[theorem]{Corollary}

\theoremstyle{definition}
\newtheorem{definition}[theorem]{Definition}

\newtheorem{question}[theorem]{Question}

\title{Why Y-c.c.}

\author{David Chodounsk{\' y}}
\author{Jind{\v r}ich Zapletal}

\address{Institute of Mathematics of the Academy of Sciences
  of the Czech Republic, \v{Z}itn\'a 25, CZ - 115 67 Praha
  1, Czech Republic}
\thanks{The work of the
first author was supported by the GACR project I 1921-N25 and RVO:
67985840}

\email{chodounsky@math.cas.cz}

\address{Institute of Mathematics of the Academy of Sciences
  of the Czech Republic, \v{Z}itn\'a 25, CZ - 115 67 Praha
  1, Czech Republic and Department of Mathematics University
  of Florida, 358 Little Hall PO Box 118105 Gainesville, FL
  32611-8105, USA}
\thanks{The work of the second author was partially supported by NSF grant DMS 1161078}
\email{zapletal@math.ufl.edu}

\begin{document}
\subjclass[2000]{03E40} 

\keywords{c.c.c.\ partitions, proper forcing, forcing axiom}

\begin{abstract}
We outline a portfolio of novel iterable properties of c.c.c.\ and proper forcing notions and study its most important instantiations, Y-c.c.\ and Y-properness. These properties have interesting consequences for partition-type forcings and anticliques in open graphs.
Using Neeman's side condition method it is possible to obtain PFA variations and prove consistency results for them.
\end{abstract}

\maketitle

\section{Introduction}

A recent work of Yorioka \cite{yorioka:random} implicitly contains the following definition.

\begin{definition}
\label{yccdefinition}
A poset $P$ satisfies \emph{Y-c.c.\ }if for every countable elementary submodel $M\prec H_\theta$ containing $P$ and every condition $q\in P$
there is a filter $F\in M$ on the completion $RO(P)$ such that $\{p\in RO(P)\cap M\colon p\geq q\}\subset F$.
\end{definition}

\noindent We will show that this is a property intermediate between $\gs$-centered and c.c.c.\ whose verification follows typical $\Delta$-system arguments. Y-c.c.\ holds for many natural examples, such as Aronszajn tree specialization forcing (Corollary~\ref{aronszajncorollary}), gap specialization forcing (Corollary~\ref{gapcorollary}), Todorcevic's posets used for the resolution of Horn--Tarski problem (Corollary~\ref{todorceviccorollary}) and other partition style c.c.c.\ posets. Y-c.c.\ has a number of pleasing consequences, such as not adding random reals (Corollary~\ref{covercorollary}) and branches into $\gw_1$-trees (Corollary~\ref{branchcorollary}), preserving uncountable chromatic number of open graphs and not adding uncountable anticliques for them (Corollary~\ref{cliquecorollary}).  It is preserved under the finite support iteration (Theorem~\ref{fsitheorem}), which means that the forcing axiom for Y-c.c.\ posets can be forced via a Y-c.c.\ poset (Corollary~\ref{ymatheorem}). There is also a non-c.c.c.\ variant:

\begin{definition}
A poset $P$ is \emph{Y-proper} if for every countable elementary submodel $M\prec H_\theta$ containing $P$ and every condition $p\in P\cap M$ there is $q\leq p$ (a \emph{Y-master condition}) which is master for $M$ and such that for every $r\leq q$
there is a filter $F\in M$ on the completion $RO(P)$ such that $\{s\in RO(P)\cap M\colon s\geq r\}\subset F$.
\end{definition}

Again, a number of natural posets are Y-proper, such as the Laver forcing (Theorem~\ref{firsttheorem}), the ideal based forcings (Theorem~\ref{idealbasedtheorem}), or the PID forcings (Theorem~\ref{pidtheorem}). Implications of Y-properness are similar to Y-c.c.\ Preservation of Y-properness under the countable support iteration is unclear, but still the Neeman method allows one to produce a Y-proper forcing which forces the forcing axiom YPFA for Y-proper posets--Theorem~\ref{ypfaforcingtheorem}. YPFA does not imply OCA.

Both Y-c.c.\ and Y-properness are instances of a wide-ranging portfolio of iterable forcing properties quite distinct from those
considered so far in the literature. The general scheme for these properties results from replacing the requirement that the sets
$F\subset RO(P)$ be filters with some other regularity demand on $F$. Section~\ref{generalsection} provides the fairly involved general iteration theorems for the resulting concepts.

The present paper owes a great deal to previous work of Yorioka. In particular, most of the examples were known to Yorioka, with proofs that inspired our proofs. Our contribution consists of isolating abstract, axiomatically useful classes of partial orders, and proving general preservation and forcing axiom theorems about them.

We use set theoretic notational standard of \cite{jech:newset}. The forcing notation follows the western convention: $p\geq q$ means that $q$ is stronger or more informative than $p$. If $P$ is a (separative) partial order then $RO(P)$ denotes the completion of $P$,
the unique complete Boolean algebra in which $P$ is dense. If $B$ is a complete Boolean algebra and $\phi$ is a statement of its forcing language, $\|\phi\|$ denotes the Boolean value of $\phi$ in $B$; that is, $\|\phi\|$ is the supremum of all $b\in B$ such that $b\Vdash\phi$. In all arguments, $\theta$ denotes a large enough regular cardinal, and $H_\theta$ the collection of all sets
whose transitive closure has size less than $\theta$. OCA denotes the Open Coloring Axiom
\cite[Section 8]{todorcevic:partitions}, the statement that every open graph on a second countable space is either countably chromatic or else contains an uncountable clique.

\section{Y-c.c.: consequences}
\label{asection}

With a novel property such as Y-c.c., it appears necessary to explore its basic consequences.

\begin{theorem}
\label{ccctheorem}
$\gs$-centeredness implies Y-c.c.\ implies c.c.c.
\end{theorem}

\begin{proof}
Suppose first that a poset $P$ is $\gs$-centered, and fix a covering $\{F_n:n\in\gw\}$ of $RO(P)$ by countably many filters. Let $M\prec H_\theta$ be a countable elementary submodel containing $P, F_n$ for every $n\in\gw$.
and let $q\in P$ be an arbitrary condition. There must be $n\in\gw$ such that $q\in F_n$; clearly, the filter $F_n\in M$ has the properties required in Y-c.c. The Y-c.c.\ of $P$ has been verified.

The other implication is more involved. 
Suppose for contradiction that $P$ is a Y-c.c.\ poset and $A\subset P$ is an antichain of size $\aleph_1$. Let $M\prec H_\theta$ be a countable
elementary submodel containing $P, A$. Let $q\in A\setminus M$ be any element and let $F\in M$
be the filter guaranteed by Y-c.c.\  Let $G=\{B\subset A\colon \sum B\in F\}$; so $G\in M$ is a collection
of subsets of $A$.

\begin{claim}
For every set $B\subset A$ in $M$, $B\in G\liff q\in B$.
\end{claim}

\begin{proof}
Suppose that $q\in B$. Then $\sum B\geq q$ is an element of $M$ which must belong to $F$ by the choice of $F$,
and so $B\in G$.

 Suppose that $q\notin B$ and for contradiction assume that $B\in G$. By the previous paragraph, $A\setminus B\in G$,
and so both $\sum B$ and $\sum (A\setminus B)$ belong to the filter $F$. This is a contradiction--the conjunction of the two is zero since $A$ is an antichain.
\end{proof}

We can now argue that $G$ is a nonprincipal $\gs$-complete ultrafilter on $A$. By the elementarity of $M$, it is enough to show
that $G$ is closed under countable intersections in $M$, and that for every set $B\subset A$ in $M$, exactly one of $B\in G$, $A\setminus B\in G$. Both of these statements follow immediately from the claim.

However, in ZFC there are no nonprincipal countably complete ultrafilters on sets of size $\aleph_1$, the final contradiction.
\end{proof}

An important feature of Y-c.c.\ posets is their interaction with open graphs. This is encapsulated in the following theorem.

\begin{theorem}
\label{cliquetheorem}
Suppose that $X$ is a second countable topological space and $H\subset X^\gw$ is an open set. If $P$ is Y-c.c.
then  every $H$-anticlique in the extension is covered by a ground model countable set of $H$-anticliques.
\end{theorem}

\noindent Here, an $H$-anticlique is just a set $A\subset X$ such that $A^\gw\cap H=0$.

\begin{proof}
Let $X$ be the space and $H\subset X^\gw$ be an open set.  Suppose that $P\Vdash \dot A\subset X$ is an anticlique. 
For every filter $F\subset RO(P)$, let $B(\dot A, F)=\{x\in X:$ for every open neighborhood $O\subset X$ of $x$,
the Boolean value $\|\check O\cap\dot A\neq 0\|$ is in the filter $F\}$.

\begin{claim}
The set $B(\dot A, F)$ is an $H$-anticlique.
\end{claim}

\begin{proof}
For contradiction assume that this fails and let $\langle x_n:n\in\gw\rangle\in H$ be a sequence of points in $B(\dot A, F)$.
Use the fact that $H$ is open to find a number $l\in\gw$ and open sets $O_k\subset X$ for $k\in l$ such that $x_k\in O_k$ and $\prod_{k\in l}O_k\times X^\gw\subset H$.
By the definition of the set $B(\dot A, F)$, the Boolean values $\|\check O_k\cap\dot A\neq 0\|$ for $k\in l$ are all in the filter $F$ and have a lower bound $p\in P$. But then, $p\Vdash (\prod_{k\in l}O_k\times X^\gw)\cap \dot A^\gw\neq 0$ so $\dot A$ is not an $H$-anticlique. This is a contradiction.
\end{proof}

Now, let $M\prec H_\theta$ be a countable elementary submodel containing $P, X, H, \dot A$; we claim that $\dot A$ is forced to be covered by 
the anticliques in the model $M$. Suppose that this fails, and let $q\in P$ be a condition and $x\in X$ a point such that
$x$ belongs to no anticliques in the model $M$, and $q\Vdash\check x\in\dot A$. Let $F\subset RO(P)$ be a filter in the model
$M$ containing all elements of $RO(P)\cap M$ weaker than $q$. Since the space $X$ is second countable, it has a basis all of whose
elements belong to the model $M$. For every such basic open set $O\subset X$ containing the point $x$, the Boolean value $\|\check O\cap \dot A\neq 0\|$
is weaker than $q$, and it belongs to the model $M$. Therefore, $x\in B(\dot A, F)$ which an anticlique in the model $M$ by the claim,
a contradiction to the choice of the point $x$.  
\end{proof}

Theorem~\ref{cliquetheorem} has a number of prominent corollaries. The following is immediate:

\begin{corollary}
\label{cliquecorollary}
Let $P$ be a Y-c.c.\ poset. Let $X$ be a second countable space and let $H\subset X^\gw$ be an open set.

\begin{enumerate}
\item If in the extension $X$ is covered by countably many $H$-anticliques, then already in the ground model it is covered by countably many anticliques;
\item if in the extension $H$ has an uncountable anticlique, then $H$ has an uncountable anticlique in the ground model.
\end{enumerate}
\end{corollary}

\noindent Thus, Y-c.c.\ posets cannot be used to force an instance of OCA in clopen graphs.

\begin{corollary}
\label{covercorollary}
Let $P$ be a Y-c.c.\ poset. If $X$ is a compact Polish space and $C$ is an $\gw_1$-cover consisting of $G_\gd$-sets, then in the extension $C$ remains an $\gw_1$-cover.
\end{corollary}

\noindent Here, a set $C\subset\power(X)$ is an $\gw_1$-cover if every countable subset of $X$ is a subset of one element of $C$. Corollary~\ref{covercorollary} needs to be understood in the context of interpretations of descriptive set theoretic notions in generic extensions: the space $X$ as well as the $G_\gd$ elements of the cover $C$ are naturally interpreted in the $P$-extension as a compact Polish space and its $G_\gd$-subsets again.

\begin{proof}
Let $p\in P$ be a condition and $\dot x_n$ for $n\in\gw$ be names for elements of $X$; we must find a set $B\in C$ and a condition $q\leq p$ such that $q\Vdash\{\dot x_n:n\in\gw\}\subset\dot B$.

\begin{claim}
There is a condition $q\leq p$ and a countable set $\{y_n:n\in\gw\}$ such that for every compact set $K\subset X$,
$$q\Vdash \dot K\cap \{\dot x_n:n\in\gw\}\neq 0\text{ implies }K\cap \{y_n:n\in\gw\}\neq 0.$$
\end{claim}

\begin{proof}
Consider the set $H\subset (K(X))^\gw$ consisting of all sequences $\langle K_n:n\in\gw\rangle$ such that $\bigcap_nK_n=0$.
A compactness argument shows that if the hyperspace $K(X)$ of compact subsets of $X$ is equipped with the Polish Vietoris topology, the set $H$ is open. For each $n\in\gw$
let $\dot A_n$ be the $P$-name for the collection of compact ground model subsets of $X$ which (whose canonical interpretations) contain the point $\dot x_n$. Clearly, $\dot A_n\subset K(X)^V$ is forced to be an $H$-anticlique. By Theorem~\ref{cliquetheorem},
there is a condition $q\leq p$ and a countable set $\{D_n\colon n\in\gw\}$ of $H$-anticliques such that $q\Vdash\bigcup_n\dot A_n\subset\bigcup_n D_n$. A compactness argument shows that the intersection of each $H$-anticlique is nonempty,
and for each $n\in\gw$ there is a point $y_n\in\bigcap D_n$. It is immediate that the set $\{y_n:n\in\gw\}$ works.
\end{proof}

Pick a condition $q\leq p$ and a set $\{y_n:n\in\gw\}$ as in the claim. Let $B\in C$ be a $G_\gd$-set such that
$\{y_n:n\in\gw\}\subset B$. It will be enough to prove that $q\Vdash\{\dot x_n:n\in\gw\}\subset\dot B$.
Suppose that this fails. As $B$ is $G_\gd$, there must be a ground model open superset of $B$ not containing the set $\{\dot x_n:n\in\gw\}$. Let $K\subset X$ be the compact complement of this open set. Then
$K\cap\{ \dot x_n:n\in\gw\}\neq 0$ while $K\cap \{y_n:n\in\gw\}=0$, a contradiction.
\end{proof}

The corollary has numerous consequences: Y-c.c.\ posets do not add random reals since the $G_\gd$ Lebesgue null sets form an $\gw_1$-cover. Y-c.c.\ posets do not separate gaps of uncountable cofinality, since each such gap induces a natural $\gw_1$-cover of $G_\gd$-sets.

Theorem~\ref{cliquetheorem} did not use the fact that the filters $F$ of Definition~\ref{yccdefinition} come from the model $M$; it was enough to assume that they come from some fixed countable set of filters on $RO(P)\cap M$. The assumption that the filters come from the model $M$ is used in the preservation of Y-c.c.\ under the finite support iteration, as well as in the following two features.

\begin{theorem}
Suppose that $P$ has Y-c.c.\ and $\kappa$ is a cardinal.  For every function $f\in\kappa^\kappa$ in the $P$-extension,
if $f\restriction a$ is in the ground model for every ground model countable set $a$, then $f$ is in the ground model.
\end{theorem}

\begin{proof}
Let $p\in P$ be a condition and $\dot f$ a name such that $p\Vdash\dot f\in\kappa^\kappa$ is a function
such that $\dot f\restriction\check a\in V$ for every countable set $a\in V$. We must find a condition $q\leq p$
and a function $g\in\kappa^\kappa$ such that $q\Vdash\check g=\dot f$.

Let $M$ be a countable elementary submodel of $H_\theta$ containing $P, p, \dot f, \kappa$. Let $q\leq p$ be a condition
deciding all values of $\dot f(\ga)$ for $\ga\in\kappa\cap M$. We will show that there is a function $g$ in the model $M$
such that $q\Vdash \dot f=\check g$.

Let $F\subset RO(P)$ be a filter in the model $M$ obtained by an application of Y-c.c.\ to $M,q$. 
By the c.c.c.\ of $P$, for every ordinal $\ga\in\kappa\cap M$, the ordinal $\gb$ such that $q\Vdash\dot f(\check\ga)=\check \gb$
must be in the model $M$. Therefore, the Boolean value $\|\dot f(\check\ga)=\check\gb\|$ is in the model $M$, it is weaker than $q$, and therefore belongs to the filter $F$.
Let $g=\{\langle\ga, \gb\rangle\in\kappa\times\kappa\colon \|\dot f(\check \ga)=\check\gb\|\in F\}$. Since $F$ is a filter, this is a partial function from $\kappa$ to $\kappa$. By the elementarity of the model $M$, $g\in M$.
We just argued that $g$ is defined for every ordinal $\ga\in M$, and so by the elementarity of the model $M$, $g$ is a total function from $\kappa$ to $\kappa$. We have also argued that $q\Vdash\dot f\restriction M=\check g\restriction M$, and by the elementarity of $M$ and the c.c.c.\ of $P$, $q\Vdash\dot f=\check g$ as desired.
\end{proof}

\begin{corollary}
\label{branchcorollary}
If $P$ has Y-c.c., then $P$ does not add any new cofinal branches into $\gw_1$-trees.
\end{corollary}

It is well known that an atomless $\gs$-centered poset adds an unbounded real, and the proof translates with the obvious changes to Y-c.c.\ posets.

\begin{theorem}
\label{unboundedtheorem}
If $P$ is an atomless poset satisfying Y-c.c., then $P$ adds an unbounded real.
\end{theorem}

Together with the preservation of Y-c.c.\ under complete subalgebras, this reproves the fact that Y-c.c.\ posets add no random reals. If an Y-c.c.\ poset did add a random real, then the measure algebra would be Y-c.c.\ which contradicts Theorem~\ref{unboundedtheorem}.

\begin{proof}
Let $M\prec H_\theta$ be a countable elementary submodel. Let $\langle F_i\colon i\in\gw\rangle$ be an enumeration of all ultrafilters on $RO(P)$ that belong to the model $M$. Each of them is nowhere dense in $RO(P)$ and so one can find a maximal antichain $A_i\subset RO(P)\setminus F_i$ in the model $M$ for every $i\in\gw$. The antichain is infinite and countable by the c.c.c.\ of $P$.
Let $\{a_i^j\colon j\in\gw\}$ be an enumeration of $A_i$ for each $i\in\gw$ and define the name $\tau$ for an element of $\baire$ by $\tau(i)=j$ if $a_i^j$ belongs to the generic filter. We claim that this is a name for an unbounded real.

Suppose not, and find a condition $q\leq p$ such that for every $i\in \gw$, $q$ is compatible with only finitely many elements of the antichain $A_i$. Let $F\subset RO(P)$ be a filter in $M$ granted by the application of Y-c.c.\ to $M, q$. Use the axiom of choice in $M$ to find $i\in\gw$ such that $F\subset F_i$. Let $B\subset A_i$ be the finite set of all elements of $A_i$ compatible with $q$.
Thus, $B\in M$, $\sum B\in M$, and necessarily $\sum B\geq q$ since no elements of $A_i\setminus B$ are compatible with $q$.
Now, $B\cap F_i=0$ and so $\sum B\notin F_i$ as $F_i$ is an ultrafilter. On the other hand, $\sum B\geq q$ and so $\sum B\in F\subset F_i$ by the choice of $i$. This is a contradiction.
\end{proof}

The existence of unbounded reals in Y-c.c.\ extensions can be derived also abstractly from Theorem~\ref{cliquetheorem} and the following argument.

\begin{theorem}
Let $P$ be a bounding poset adding a new point $\dot x\in\cantor$. Then

\begin{enumerate}
\item either some condition forces $\dot x$ to be c.c.c.\ over the ground model and then some $\gw_1$-cover of $G_\gd$ sets on a compact Polish space is not preserved;
\item or $\dot x$ is forced not to be c.c.c.\ over the ground model, and then there is a compact Polish space $X$, an open graph $H\subset [X]^2$ and an $H$-anticlique in the extension which is not covered by countably many ground model $H$-anticliques.
\end{enumerate}
\end{theorem}

\noindent Here, a point $x\in\cantor$ is c.c.c.\ over the ground model if there is a $\gs$-ideal $I$ on $\cantor$ in the ground model which is c.c.c.\ (i.e.\ there is no uncountable collection of Borel pairwise disjoint $I$-positive sets) and $x$ belongs to no Borel set in $I$
coded in the ground model.

\begin{proof}
Suppose that $p\in P$ is a condition. Let $I_p$ be the $\gs$-ideal on $\cantor$ consisting of all analytic sets $A\subset\cantor$ such that $p\Vdash\dot x\notin\dot A$. 

\begin{claim}
Every $I_p$-positive analytic set has an $I_p$-positive compact subset.
\end{claim}

\begin{proof}
Let $A\notin I_p$ be an analytic set, and let $T\subset (2\times\gw)^{<\gw}$ be a tree such that $A=\proj([T])$. Let $q\leq p$ be a condition forcing $\dot x\in\dot A$, and let $\dot y$ be a name for a function in $\baire$ such that $q\Vdash\langle \dot x, \dot y\rangle\in [\check T]$. Use the bounding assumption to find a condition $r\leq q$ and a function $z\in\baire$ such that $r\Vdash\dot y$ is dominated by $\check z$. Let $S$ be the tree obtained from $T$ by erasing all nodes which exceed the function $z$ at some point in their domain. Then $S$ is a finitely branching tree, $\proj([S])\subset A$ is compact, $r\Vdash \dot x\in \proj([S])$ and
so $\proj([S])$ is $I_p$-positive. The claim follows.
\end{proof}

Now, suppose that there is a c.c.c.\ $\gs$-ideal $J$ and a condition $p\in P$ which forces that $x$ belongs to no Borel set in $J$. Since each $I_p$-positive Borel set is $J$-positive, the $\gs$-ideal $I_p$ is c.c.c. Every Borel set $B\in I_p$ is covered by a $G_\gd$ set $C\in I_p$, namely $C=\cantor\setminus\bigcup A$ where $A$ is some (countable) maximal antichain of compact $I_p$-positive subsets of $\cantor$ disjoint from $B$. It follows that
the collection of all $G_\gd$ sets in the ideal $I_p$ is a $\gw_1$-cover, and the condition $p$ forces it not to be an $\gw_1$-cover in the extension--none of its elements contain the point $\dot x$.

Suppose on the other hand that $x$ is forced not to be c.c.c., and the ideal $I_p$ is not c.c.c.\ for any condition $p\in P$. Let $X=K(\cantor)$ and consider the open graph $H\subset X^2$ consisting of all pairs $\langle K, L\rangle$ such that $K\cap L=0$. Let $\dot A$ be a name for the $H$-anticlique consisting of all compact sets in the ground model containing the point $\dot x$. We claim that it is forced not to be covered by countably many anticliques in the ground model. Suppose that $p\in P$ is a condition and $B_n$ for $n\in\gw$ are $H$-anticliques;
we will find a condition $q\leq p$ and a compact set $K\subset\cantor$ such that $K\notin\bigcup_nB_n$ and $q\Vdash\dot x\in\dot K$. Just observe that the $\gs$-ideal $I_p$ is not c.c.c.\ and use the claim to produce an uncountable collection $C$ of pairwise disjoint $I_p$-positive compact sets. Note that this collection is an $H$-clique, and therefore for each $n\in\gw$ the intersection $B_n\cap C$ can contain at most one set. As $C$ is uncountable, there must be $K\in C\setminus\bigcup_nB_n$ and then any
condition $q\leq p$ forcing $\dot x\in\dot K$ is as required.
\end{proof}

\section{Y-c.c.: examples}

Y-c.c.\ in all of our examples is verified using the same general theorem.

\begin{theorem}
\label{ctheorem}
Let $P$ be a poset. Suppose that there is a function $w$ defined on $P$ such that

\begin{enumerate}
\item for every $p\in P$, $w(p)$ is a finite set;
\item if $p, q\in P$ are compatible, then they have a lower bound $r\leq p,q$ such that $w(r)=w(p)\cup w(q)$;
\item whenever $\{p_\ga\colon \ga\in\gw_1\}$ and $\{q_\ga\colon\ga\in\gw_1\}$ are subsets of $P$ such that $\{w(p_\ga)\colon\ga\in\gw_1\}$ and $\{w(q_\ga)\colon\ga\in\gw_1\}$ are $\Delta$-systems with the same root, then
there are ordinals $\ga, \gb\in\gw_1$ such that $p_\ga$ and $q_\gb$ are compatible.
\end{enumerate}

\noindent Then $P$ is Y-c.c.
\end{theorem}

\begin{proof}
Let $a$ be a finite set. Say that a set $A\subset P$ is $a$-large
if for every countable set $b\supset a$ there is a condition $p\in A$ such that $w(p)\cap b=a$.

\begin{claim}
\label{demclaim}
The set $\{\sum A\colon A$ is $a$-large$\}\subset RO(P)$ is centered.
\end{claim}

\noindent The trivial case where there are no $a$-large sets at all is included in the statement of the claim.

\begin{proof}
Let $\{A_i\colon i\in n\}$ be finitely many $a$-large sets; we must produce a condition $q\in P$ which for each $i$ has an element of $A_i$ above it. To this end, use transfinite induction and the largeness assumption to find conditions $p_\ga^i\in A_i$ for each $\ga\in\gw_1$ and $i\in n$ so that $\{w(p_\ga^i)\colon\ga\in\gw_1\}$ is a $\Delta$-system with root $a$ for each $i\in n$.

By induction on $i\in n$ find sets $\{q_\ga^i\colon\ga\in\gw_1\}$ so that $\{w(q_\ga^i)\colon \ga\in\gw_1\}$ forms a $\Delta$-system with root $a$, and for every $\ga\in\gw_1$ and every $j\in i+1$ there is $\gb\in\gw_1$ such that $q_\ga^i\leq p_\gb^j$. The step $i=0$ is trivially satisfied with $p_\ga^i=q_\ga^i$. To perform the induction step, by transfinite recursion on $\gg\in\gw_1$ use item (3) of the assumptions repeatedly to find countable ordinals $\ga_\gg$ and $\gb_\gg$ such that $q_{\ga_\gg}^i$ and $p_{\gb_\gg}^i$ are compatible and whenever $\gd\neq\gg$ then $(w(p_{\gb_\gg}^i)\cup w(q_{\ga_\gg}^i))\cap (w(p_{\gb_\gd}^i)\cup w(q_{\ga_\gd}^i))=a$. Then, use (2) to find conditions $q_\gg^{i+1}\leq p_{\gb_\gg}^i, q_{\ga_\gg}^i$ such that $w(q_{\gg}^{i+1})=w(p_{\gb_\gg}^i)\cup w(q_{\ga_\gg}^i)$; this concludes the induction step.

In the end the condition $q=q_0^{n-1}$ works as required.
\end{proof}

Now suppose that $M\prec H_\theta$ is a countable elementary submodel containing $w, P$, and suppose that $q\in P$ is any condition. Let $a=M\cap w(q)\in M$, and find a filter $F\in M$ on $RO(P)$ extending the centered system $\{\sum A\colon A$ is $a$-large$\}$. We will show that for every $p\in M\cap RO(P)$, if $p\geq q$ then $p\in F$. This will complete the proof of Y-c.c.\ for $P$.

Let $A=\{r\in P\colon r\leq p\}\in M$. It will be enough to show that $A$ is $a$-large, since $P\subset RO(P)$ is dense, and so $\sum A=p$ and $p\in F$. Suppose for contradiction that $A$ is not $a$-large. A counterexample, a countable set $b$, can be found in the model $M$ by elementarity. But then, the condition $q\in A$ satisfies $w(q)\cap b=a$, contradicting the assumption that $b$ is a counterexample. Thus, the poset $P$ is Y-c.c.
\end{proof}

The first specific example of a Y-c.c.\ forcing is the specialization forcing for a tree without branches of length $\gw_1$. Let $T$ be a such a tree and consider the specialization poset $P(T)$ consisting of all finite functions $p:T\to\gw$ such that for all $s<t$ in $\dom(p)$ the values $p(s), p(t)$ are distinct. The ordering is that of reverse inclusion.

\begin{corollary}
\label{aronszajncorollary}
If $T$ is a tree without branches of length $\gw_1$, then the specialization forcing $P(T)$ satisfies Y-c.c.
\end{corollary}

\begin{proof}
For every condition $p\in P(T)$ let $w(p)=\dom(p)$. It will be enough to show that item (3) of Theorem~\ref{ctheorem} is satisfied. 
Suppose that $\{p_\ga\colon\ga\in \gw_1\}$ and $\{q_\ga\colon\ga\in\gw_1\}\subset P(T)$ are sets such that $\{w(p_\ga)\colon\ga\in \gw_1\}$ and $\{w(q_\ga)\colon\ga\in\gw_1\}$ are $\Delta$-systems with the same root $a$. By transfinite recursion on $\gg\in\gw_1$ find ordinals $\ga_\gg$ and $\gb_\gg$ such that the sets $b_\gg=(\dom(p_{\ga_\gg})\cup \dom(q_{\gb_\gg}))\setminus a$ for $\gg\in\gw_1$ are pairwise disjoint and contain no element of $T$ which is below some element of $a$. 

For each $\gg\in\gw_1$ consider a condition $r_\gg\in P(T)$ whose domain is the set of minimal elements of $b_\gg$ and which assigns to every element of its domain value $0$. Since the poset $P(T)$ is c.c.c.\ (see e.g.~\cite{baumgartner:thesis}) there must be countable ordinals $\gd\neq\gg$ for which $r_\gg$ and $r_\gd$ are compatible, i.e.\ no elements of $b_\gg$ is compatible with any element of $b_\gd$. It is immediate that the conditions $p_{\ga_\gg}$ and $q_{\gb_\gd}\in P$ are compatible as well.
\end{proof}

The usual gap specialization forcing satisfies Y-c.c.\ as well. To this end, recall basic definitions. An $(\gw_1, \gw_1)$-pregap is a
sequence $\langle x_\ga, y_\ga\colon \ga\in\gw_1\rangle$ of subsets of $\gw$ such that $x_\ga\cap y_\ga=0$ and $\gb\in\ga$ implies that $x_\gb\subset^* x_\ga$ and $y_\gb\subset^* y_\ga$, each time up to finitely many exceptional natural numbers. A set $c\subset\gw$ separates the pregap if for every ordinal $\ga\in\gw_1$, $x_\ga\subset^* c$ and $y_\ga\cap c=^*0$ holds.
A gap is a pregap that cannot be separated.

A pregap is a gap if and only if for every uncountable set $D\subset\gw_1$ there are distinct ordinals $\ga, \gb\in D$ such that
$(x_\ga\cap y_\gb)\cup (x_\gb\cap y_\ga)\neq 0$. A gap is special if there is an uncountable set $D\subset\gw_1$ such that for all distinct ordinals $\ga, \gb\in D$
$(x_\ga\cap y_\gb)\cup (x_\gb\cap y_\ga)\neq 0$ holds. For a special gap it is impossible to introduce a separating set without collapsing $\gw_1$.

There is a natural specializing forcing for gaps \cite[Lemma 3.8]{bekkali:set}. Suppose that $H=\langle x_\ga, y_\ga\colon\ga\in\gw_1\rangle$ is a gap. Let $P(H)$ be the poset of all finite sets $p\subset\gw_1$ such that for distinct ordinals $\ga, \gb\in p$ the requirement 
$(x_\ga\cap y_\gb)\cup (x_\gb\cap y_\ga)\neq 0$ holds. It turns out that $P(H)$ is c.c.c. It follows that there is a condition $p\in P(H)$ which forces that the union of the generic filter is uncountable; this is the specializing set.

\begin{corollary}
\label{gapcorollary}
If $H$ is a $(\gw_1, \gw_1)$-gap then the specialization forcing $P(H)$ satisfies Y-c.c.
\end{corollary}

\begin{proof}
For every condition $p\in P(H)$ let $w(p)=p$. It will be enough to show that item (3) of Theorem~\ref{ctheorem} is satisfied. 
Suppose that $\{p_\ga\colon\ga\in \gw_1\}$ and $\{q_\ga\colon\ga\in\gw_1\}\subset P(H)$ are sets such that $\{w(p_\ga)\colon\ga\in \gw_1\}$ and $\{w(q_\ga)\colon\ga\in\gw_1\}$ are $\Delta$-systems with the same root $a$. By transfinite recursion on $\gg\in\gw_1$ find ordinals $\ga_\gg$ and $\gb_\gg$ such that the sets $b_\gg=(p_{\ga_\gg}\cup q_{\gb_\gg})\setminus a$ for $\gg\in\gw_1$ are pairwise disjoint and contain no ordinals less or equal to $\max(a)$. 

 Use a counting argument to find an uncountable set $D\subset\gw_1$ and a number $k\in\gw$ such that for every
ordinal $\gg\in D$, the sets $\{x_\gd\setminus k\colon\gd\in b_\gg\}$ and the sets $\{y_\gd\setminus k\colon\gd\in b_\gg\}$
are linearly ordered by inclusion. For each $\gg\in D$, write $\gd_\gg=\min(b_\gg)$ and let $c_\gg=x_{\gd_\gg}\setminus k$ and $d_\gg=y_{\gd_\gg}\setminus k$. The object $\langle c_\gg, d_\gg\colon\gg\in D\rangle$
is a gap since any set separating this gap would also separate the original gap. Therefore, there must be ordinals $\gg\neq\gg'\in D$
such that $(c_\gg\cap d_{\gg'})\cup (c_{\gg'}\cap d_\gg)\neq 0$. It is easy to verify that the conditions $p_{\ga_\gg}, q_{\gb_{\gg'}}\in P(H)$ are compatible as required.
\end{proof}

Todorcevic \cite[Theorem 7.8]{todorcevic:partitions} introduced a partition-type forcing associated with unbounded sequences of functions of length $\gw_1$; this poset satisfies Y-c.c.\ as well. Let $\vec f=\langle f_\ga\colon\ga\in\gw_1\rangle$ be a modulo finite increasing, unbounded sequence of increasing functions in $\baire$.
Let $P(\vec f)$ be the poset of all finite sets $p\subset\gw_1$ such that for all ordinals $\ga\in\gb$ in the set $p$, there is $n$ such that $f_\ga(n)>f_\gb(n)$. The ordering of $P(\vec f)$ is that of inclusion.

\begin{corollary}
\label{unboundedcorollary}
If the sequence $\vec f$ is unbounded then the poset $P(\vec f)$ satisfies Y-c.c.
\end{corollary}

\begin{proof}
For every condition $p\in P(\vec f)$ let $w(p)=p$. It will be enough to show that item (3) of Theorem~\ref{ctheorem} is satisfied. 
Suppose that $\{p_\ga\colon\ga\in \gw_1\}$ and $\{q_\ga\colon\ga\in\gw_1\}\subset P(\vec f)$ are sets such that $\{w(p_\ga)\colon\ga\in \gw_1\}$ and $\{w(q_\ga)\colon\ga\in\gw_1\}$ are $\Delta$-systems with the same root $a$. By transfinite recursion on $\gg\in\gw_1$ find ordinals $\ga_\gg$ and $\gb_\gg$ such that the sets $b_\gg=(p_{\ga_\gg}\cup q_{\gb_\gg})\setminus a$ for $\gg\in\gw_1$ are pairwise disjoint and contain no ordinals less or equal to $\max(a)$. 

 Use a counting argument to find an uncountable set $D\subset\gw_1$ and a number $k\in\gw$ such that for every
ordinal $\gg\in D$
the functions $\{f_\gd\colon\gd\in b_\gg\}$ are linearly ordered by domination everywhere above $k$.
Let $\gd_\gg=\min(b_\gg)$. The collection $\langle f_{\gd_\gg}\colon\gg\in\gw_1\rangle$ is unbounded, and therefore there is a number
$n>k$ such that for every $m$ there is an ordinal $\gg(m)\in D$ such that $f_{\gd_{\gg(m)}}(n)>m$. Let $\gg'\in D$ be an ordinal larger than all $\gg(m)$ for $m\in\gw$, let $m=\max\{f_\gd(n)\colon \gd\in b_{\gg'}\}$ and observe that the conditions $p_{\ga_{\gg(m)}}, q_{\gg'}$ are compatible as desired.
\end{proof}

Balcar, Paz{\' a}k, and Th{\" u}mmel \cite{bpt:todorcevic}, following Todorcevic \cite{todorcevic:forcing}, defined an ordering $T(Y)$ for every topological space $Y$. Th{\" u}mmel used these orderings to settle an old problem of Horn and Tarski
\cite{thuemmel:horntarski}. There are several closely related definitions of $T(Y)$; we will use the following. $T(Y)$ consists of all sets $p\subset Y$ such that
$p$ is a union of finitely many converging sequences together with their limits. For $p\in T(Y)$, write $w(p)$ for the set of its accumulation points. The ordering is defined by $q\leq p$ if $p\subset q$ and $w(q)\cap p=w(p)$.

The poset $T(Y)$ may or may not be c.c.c., $\gs$-centered etc, depending on the topological space $Y$. However,
if it is c.c.c., then it automatically assumes Y-c.c. This improves a result of Yorioka \cite{yorioka:random}.

\begin{corollary}
\label{todorceviccorollary}
Let $Y$ be a topological space. If $T(Y)$ is c.c.c., then $T(Y)$ has Y-c.c.
\end{corollary}

\begin{proof}
We will show that the function $p\mapsto w(p)$ has the required properties. It will be enough to show that item (3) of Theorem~\ref{ctheorem} is satisfied. 
Suppose that $\{p_\ga\colon\ga\in \gw_1\}$ and $\{q_\ga\colon\ga\in\gw_1\}\subset T(Y)$ are sets such that $\{w(p_\ga)\colon\ga\in \gw_1\}$ and $\{w(q_\ga)\colon\ga\in\gw_1\}$ are $\Delta$-systems with the same root $a$. By transfinite recursion on $\gg\in\gw_1$ find ordinals $\ga_\gg$ and $\gb_\gg$ such that the sets $b_\gg=(w(p_{\ga_\gg})\cup w(q_{\gb_\gg}))\setminus a$ for $\gg\in\gw_1$ are pairwise disjoint.

For each $\gg\in\gw_1$ consider a condition $r_\gg=p_{\ga_\gg}\cup q_{\gb_\gg}\in T(Y)$. Since the poset $T(Y)$ is c.c.c.\ there must be countable ordinals $\gd\neq\gg$ for which $r_\gg$ and $r_\gd$ are compatible, i.e.\ no element of $w(r_\gg)$ is an isolated point of $r_\gd$ and vice versa. 
It is immediate that the conditions $p_{\ga_\gg}$ and $q_{\gb_\gd}\in P$ are compatible as well.
\end{proof}

As a final remark in this section, note that the OCA partition posets in general do not have Y-c.c.\ by Theorem~\ref{cliquetheorem}.

\begin{question}
Suppose that $I$ is a suitably definable ideal on a Polish space $X$ and let $P_I$ be the quotient Boolean algebra of Borel subsets of $X$
modulo $I$. Are the following equivalent?

\begin{enumerate}
\item $P_I$ is Y-c.c.;
\item $P_I$ is $\gs$-centered.
\end{enumerate}
\end{question}

Note that if a c.c.c.\ poset $P$ is a finite support product of posets, and each factor satisfies the assumptions of Theorem~\ref{ctheorem}, then $P$ also satisfies the assumptions, and is Y-c.c.

\begin{question}
Suppose that $P, Q$ are Y-c.c.\ posets such that $P\times Q$ is c.c.c. Must $P\times Q$ be Y-c.c.?
\end{question}

\section{Y-properness}

The proper variation of Y-c.c.\ yields much greater variety of posets. The basic consequences of Y-properness remain mostly the same
as for Y-c.c.\ and also the proofs of Section~\ref{asection} immediately adapt to give the following:

\begin{theorem}
\label{yproperpreservationtheorem}
Suppose that $P$ is a Y-proper poset.

\begin{enumerate}
\item Whenever $\kappa$ is a cardinal and $f\in \kappa^\kappa$ is a function in the $P$-extension which is not in the ground model,
then there is a ground model countable set $a\subset\kappa$ such that $f\restriction a$ is not in the ground model;
\item $P$ preserves $\gw_1$-covers consisting of $G_\gd$ sets on compact Polish spaces;
\item if $X$ is a second countable topological space and $H\subset X^\gw$ is an open set, then every $H$-anticlique in the extension is covered by countably many ground model anticliques;
\item if $P$ is atomless, then it adds an unbounded real.
\end{enumerate}
\end{theorem}

The notion of Y-properness is particularly suitable for side condition type proper forcings. We discuss two classes of examples.

\cite{z:keeping} introduced the notion of ideal based forcings. Yorioka showed that ideal based forcings do not add random reals. We will now show that ideal based forcings are Y-proper. This class of forcings includes posets used for destroying S-spaces, forcing a five-element classification of directed partial orders of size $\aleph_1$, and others.

 First, the rather involved definitions must be carefully stated.

\begin{definition}
\label{tripledefinition}
An \emph{ideal based triple} is a triple $\langle U, \pprec, I\rangle$ such that the following are satisfied for $\pprec$:

\begin{enumerate}
\item[(1)] $U$ is a collection of finite subsets of $\gw_1$ and $\pprec$ is an ordering on it refining inclusion;
\item[(2)] whenever $a\in U$ and $\gb\in\gw_1$ then $a\cap\gb\in U$ and $a\cap\gb\pprec a$;
\end{enumerate}

\noindent and the following are satisfied about $I$:

\begin{enumerate}
\item[(3)] $I$ is an ideal on $\gw_1$ including all singletons;
\item[(4)] every $I$-positive set has a countable $I$-positive subset;
\item[(5)] for every $a\in U$ the set $\{\gb\in \gw_1\colon a\pprec a\cup \{\gb\}\}$ is not covered by countably many elements of $I$;
\item[(6)] for every $a\in U$ the set $\{\gb\in\gw_1\colon a\cap\gb\pprec (a\cap \gb)\cup\{\gb\}$ and $a\not\pprec a\cup\{\gb\}\}$ is in $I$.
\end{enumerate}
\end{definition}

\begin{definition}
Given an ideal based triple $\langle U, \pprec, I\rangle$, the associated ideal based forcing
$P$ is defined as follows. A condition $p\in P$ is a finite set of ordered pairs $\langle M, \ga\rangle$ such that $M\prec H_\lambda$ is a countable elementary submodel for some fixed $\lambda$ such that $I\in H_\lambda$, $\ga$ is a countable ordinal which does not belong to $\bigcup (I\cap M)$, and

\begin{itemize}
\item $w(p)=\{\ga\colon\exists M\ \langle M, \ga\rangle\in p\}\in U$;
\item whenever $\langle M, \ga\rangle$ and $\langle N, \gb\rangle$ are distinct elements of $p$ then either
$M\in N$ and $\ga\in N$, or $N\in M$ and $\gb\in M$.
\end{itemize}

\noindent The ordering on the poset $P$ is defined by $p\geq q$ if $p\subset q$ and $w(p)\pprec w(q)$. 
\end{definition}

\begin{theorem}
\label{idealbasedtheorem}
If $P$ is an ideal based forcing then $P$ is Y-proper.
\end{theorem}

\begin{proof}
Fix the ideal based triple $\langle U, \pprec, I\rangle$ generating the poset $P$.
Let $p\in P$ be a condition. Say that a set $A\subset P$ is $p$-large if Player II has a winning strategy in the following game $G(A,p)$.
Player I starts out with a countable set $z\in H_{\lambda}$. Then, Players I and II alternate for $\gw$ many rounds, Player I starts round $k$ with a set $b_k\in I$
and Player II answers with a countable ordinal $\ga_k\notin b_k$ such that $\ga_0\in \ga_1\in\dots$
Player II wins if there is a number $l$ and a condition $q\leq p$
such that $w(q)=w(p)\cup\{\ga_k:k\in l\}$, the $\in$-least model $M$ on $q\setminus p$ contains $z$ as an element and $w(q)\cap M=w(p)$, and there is $r\in A$ such that $r\geq q$. Note that the game is open for Player II and therefore determined.

\begin{claim}
\label{aclaim}
The collection $\{\sum A:A$ is $p$-large$\}\subset RO(P)$ is centered.
\end{claim}

\begin{proof}
Let $\{A_i\colon i\in n\}$ be a collection of $p$-large sets; we must find conditions $p_i\in A_i$ for each $i\in n$ with a common lower bound.

Let $\langle M_i\colon i\in n+1\rangle$ be an $\in$-chain of countable elementary submodels of some large $H_\theta$ with $P, p, I, A_i$ for $i\in n$ all elements of $M_0$. By induction
on $i\in n$, we will construct conditions $p_i\in A_i, q_i$ such that

\begin{itemize}
\item $p, p_i$ are both weaker than $q_i$, $q_i\in M_{n-i}$, and the $\in$-least model on $q_i\setminus p$ contains $M_{n-i-1}\cap H_\lambda$ as an element;
\item for each $i$, $r_i=\bigcup_{j\in i}q_j$ is a condition in $P$ which is a lower bound of all conditions $q_j$ for $j\in i$.
Moreover, $w(r_i)\cap M_{n-i}=w(p)$.
\end{itemize}

Suppose that conditions $p_j, q_j$ have been constructed for $j\in i$. Write $a_i=w(r_i)$. Work in the model $M_{n-i}$. Let $\gs_i$ be a winning strategy for Player II in the game $G(A_i, p)$. We will produce an infinite play of the game such that 

\begin{itemize}
\item the initial move of Player I is $M_{n-i-1}\cap H_\lambda$; 
\item all moves are in the model $M_{n-i}$;
\item writing $e_l=\{\ga_k:k\in l\}$, for every $l$ we have $a_i\pprec a_i\cup e_l$.
\end{itemize}

The play is easy to construct by induction on $l\in\gw$. Suppose the first $l$ moves have been constructed, producing a play $t_l$.
Write $c=\{\ga\in\gw_1:$ for some set $b\in I$ the strategy $\gs_i$ answers the play $t_l^\smallfrown b$ with $\ga\}$.
Thus, $c\in M_{n-i}$. Note that for every ordinal $\ga\in c$, $w(p)\cup e_l\pprec w(p)\cup e_l\cup\{\ga\}$: since there is a play in which Player II wins, producing a condition $r$ such that $w(r)$ contains $w(p)\cup e_l\cup\{\ga\}$ as an initial segment, this follows from (2) of Definition~\ref{tripledefinition}. Also, $c$ is an $I$-positive set: if it were an element of $I$, then Player I could play a set containing $c$, forcing
the strategy $\gs_i$ to answer with an ordinal out of $c$, contradicting the definition of $c$. By (4) of Definition~\ref{tripledefinition}, the set $c$ contains an $I$-positive countable subset $d\subset c$, and this set $d$ can be found in the model $M_{n-i}$.
By (6), there is an ordinal $\ga_l\in d$ such that
$a_i\cup\{\ga_k:k\in l\}\pprec a_i\cup\{\ga_k:k\in l+1\}$. Find a move $b_l\in M_{n-i}$ provoking the strategy $\gs_i$ to answer $\ga_l$ and let $t_{l+1}=t_l^\smallfrown b_l^\smallfrown \ga_l$. This concludes the induction step and the construction of the play.

Now, since $\gs_i$ is a winning strategy for Player II, there is a natural number $l$ and conditions $p_i\in A_i$ and $q_i$
such that $q_i\leq p_i, p$, $q_i\in M_{n-i}$, and $w(q_i)=w(p)\cup e_l$. Consider the set $r_{i+1}=\bigcup_{j\leq i}q_j$.
The set $r_{i+1}$ is a condition in the poset $P$ smaller than $r_i$ by the third item of the construction of the infinite play. Also, $r_{i+1}\leq q_i$ by (2) of Definition~\ref{tripledefinition}. This concludes the induction step of the induction on $i$ and the proof of the claim.
\end{proof}

Now we are ready to verify Y-properness for the poset $P$. Let $M\prec H_\theta$ be a countable elementary submodel containing
$U, \pprec, I$, let $p\in P\cap M$. Find an ordinal $\ga\in\gw_1\setminus \bigcup (I\cap M)$ such that $w(p)\pprec w(p)\cup \{\ga\}$; such ordinal exists by Definition~\ref{tripledefinition}(5). Let $q=p\cup \{\langle  M\cap H_{\aleph_2}, \ga\rangle\}$. It is not difficult to see that $q\leq p$. \cite{z:keeping} shows that $q$ is a master condition for $M$. We shall show that $q$ is a Y-master condition for the model $M$.

 Suppose that $r\leq q$ is an arbitrary condition. Observe that $r\cap M\in P$ is a condition weaker than $r$. Let $F\in M$ be a filter on $RO(P)$
extending the centered system $\{\sum A:A$ is $r\cap M$-large$\}$. We claim that for every condition $s\in RO(P)\cap M$ such that
$s\geq r$, $s\in F$ holds. This will conclude the proof.

Indeed, let $A=\{t\in P:t\leq s\}$. Since $P$ is dense in $RO(P)$, it is clear that $\sum A=s$. To conclude the proof, it will be enough to show that $A$ is $r\cap M$-large. Suppose that it is not. The game $G(A, r\cap M)$ is determined, Player II has no winning strategy, therefore Player I has a winning strategy, and such a strategy $\gs$ has to exist in the model $M$ as $r\cap M, s\in M$. Note that the strategy $\gs$ is in $H_{\aleph_2}$, and so it belongs to all models on $r\setminus M$. The definition of the poset $P$ shows that Player II can defeat the strategy by playing the ordinals in $w(r)\setminus M$ in increasing order, since then the condition $r\leq s$ will witness the defeat of Player I at the appropriate finite stage. This is the final contradiction.
\end{proof}

Another class of Y-proper posets comes from the usual way of forcing the P-ideal dichotomy, PID \cite{todorcevic:pid, todorcevic:bsl}.
Let $X$ be a set and $I\subset [X]^{\leq \aleph_0}$ be a P-ideal containing all singletons. This means that for every countable set
$J\subset I$ there is a set $a\in I$ such that for every $b\in J$, $b\subset^* a$. Suppose that $X$ is not a countable union
of sets $Y_n$ for $n\in\gw$ such that $\power(Y_n)\cap I=[Y_n]^{<\aleph_0}$. Then there is a proper poset $P$ adding
an uncountable set $Z\subset X$ such that $[Z]^{\aleph_0}\subset I$, which we now proceed to define.

For simplicity assume that the underlying set $X$ is a cardinal $\kappa$. Let $K$ be the $\gs$-ideal on $X$ generated by those sets $Y\subset X$ such that $I\cap\power(Y)= [Y]^{<\aleph_0}$.
Thus, the assumptions imply that $X\notin K$.
The poset $P$ consists of conditions $p$, which are finite sets of triples $\langle M, x, a\rangle$ such that
$M\prec H_{\kappa^+}$ is a countable elementary submodel, $x\in X$ is a point which does not belong to $\bigcup (K\cap M)$,
and $a\in I$ is a set which modulo finite contains all sets in $I\cap M$. Moreover, if $\langle M, x,a\rangle$ and $\langle N, y, b\rangle$ are distinct elements of $p$, then either $M, x, a\in N$ or $N, y, b\in M$. The ordering is defined by $q\leq p$ if
$p\subseteq q$ and whenever $\langle M, x, a\rangle\in q\setminus p$ and $\langle N, y, b\rangle\in p$ are such that $M\in N$
then $x\in b$. As in the ideal-based case, for a condition $p\in P$ we write $w(p)=\{x\in X:\exists M, a\ \langle M, x, a\rangle\in p\}$.

\begin{theorem}
\label{pidtheorem}
The PID poset $P$ is Y-proper.
\end{theorem}

\begin{proof}
Suppose that $p\in P$ is a condition and $A\subset P$ is a set. Say that $A$ is $p$-large if Player II has a winning strategy in the following game $G(A, p)$. In the game, Player I starts with a set $z\in H_{\kappa^+}$, and then Player I and II alternate for $\gw$ many rounds. At round $k$, Player I plays a set $Y_k\in K$ and Player II answers with a point $x_k\in X\setminus Y_k$.
Player II wins if at some round $l\in\gw$ there are conditions $q\in P$ and $r\in A$ such that $q$ is a lower bound of $p,r$, $w(q)=w(p)\cup\{x_k\colon k\in l\}$, and the $\in$-first model $M$ on $q\setminus p$ contains the set $z$, and $w(q)\cap M=w(p)$. Note that the game is open for Player II and therefore determined.

\begin{claim}
The set $\{\sum A\colon A$ is $p$-large$\}\subset RO(P)$ is centered.
\end{claim}

\begin{proof}
Let $\{A_i\colon i\in n\}$ be a collection of $p$-large sets; we must find conditions $p_i\in A_i$ for each $i\in n$ with a common lower bound.

Let $\langle M_i\colon i\in n+1\rangle$ be an $\in$-chain of countable elementary submodels of some large $H_\theta$ with $M_0$ containing $X, I, P, p, A_i$ for $i\in n$ as elements. By induction
on $i\in n$, we will construct conditions $p_i\in A_i, q_i$ such that

\begin{itemize}
\item $p, p_i$ are both weaker than $q_i$, $q_i\in M_{n-i}$, and the $\in$-least model on $q_i\setminus p$ contains $M_{n-i-1}\cap H_{\kappa^+}$ as an element;
\item for each $i$, $r_i=\bigcup_{j\in i}q_j$ is a condition in $P$ which is a lower bound of all conditions $q_j$ for $j\in i$,
and $w(r_i)\cap M_{n-i}=w(p)$.
\end{itemize}

Suppose that conditions $p_j, q_j$ have been constructed for $j\in i$. Write $a_i\subset X$ for the intersection of all
sets in the P-ideal $I$ which occur on $r_i\setminus p$. Observe that every set in $I\cap M_{n-i}$ is contained in $a_i$ up to finitely many exceptions. Work in the model $M_{n-i}$. Let $\gs_i$ be a winning strategy for Player II in the game $G(A_i, p)$. We will produce an infinite play of the game such that 

\begin{itemize}
\item the initial move of Player I is $M_{n-i-1}\cap H_{\kappa^+}$; 
\item all moves are in the model $M_{n-i}$;
\item all moves of Player II belong to the set $a_i$.
\end{itemize}

The play is easy to construct by induction on $l\in\gw$. Suppose the first $l$ moves have been constructed, producing a play $t_l$.
Write $c=\{x\in X:$ for some set $b\in K$ the strategy $\gs_i$ answers the play $t_l^\smallfrown b$ with $x\}$.
Thus, $c\in M_{n-i}$. Observe that $c\notin K$: if $c\in K$, then Player I could play the set $c$, forcing
the strategy $\gs_i$ to answer with a point out of $c$, contradicting the definition of $c$. By the definition of the ideal $K$, the set $c$ contains an infinite countable set $d\subset c$ in the P-ideal $I$, and this set $d$ can be found in the model $M_{n-i}$.
Thus, the intersection $a_i\cap d$ is nonempty, containing some element $x\in M_{n-i}$. Find a move $b_l\in M_{n-i}$ provoking the strategy $\gs_i$ to answer with $x$ and let $t_{l+1}=t_l^\smallfrown b_l^\smallfrown x$. This concludes the induction step and the construction of the play.

Now, since $\gs_i$ is a winning strategy for Player II, there is a natural number $l$ and conditions $p_i\in A_i$ and $q_i$
such that $q_i\leq p_i, p$, $q_i\in M_{n-i}$, and all points in $X$ appearing on $q_i\setminus p$ belong to the set $a_i$. 
It is immediate to verify that $r_{i+1}=\bigcup_{j\leq i}q_i$ is a lower bound of $r_i$ and $q_i$. This concludes the induction step of the induction on $i$ and the proof of the claim.
\end{proof}

Now suppose that $M\prec H_\theta$ is a countable elementary submodel containing $X,I$, and let $p\in P\cap M$ be any condition.
We must produce a Y-master condition $q\leq p$ for the model $M$. Let $x\in X$ be some point not in $\bigcup(K\cap M)$, and let
$a\in I$ be some set which modulo finite contains all sets in $I\cap M$; these objects exist by initial assumptions on the ideal $I$.
Let $q=p\cup\{\langle M\cap H_{\kappa^+}, x, a\rangle\}$. \cite{todorcevic:pid} shows that $q$ is a master condition for $M$. We will show that $q$ is a Y-master condition for the model $M$.

Let $r\leq q$ be a condition. Note that $r\cap M\in P$ is a condition weaker than $r$. Let $F\in M$ be any filter on $RO(P)$ extending the centered system $\{\sum A:A\subset P$ is $r\cap M$-large$\}$. We will show that for every condition $s\in RO(P)\cap M$, if $s\geq r$ then $s\in F$. This
will conclude the proof.

Indeed, suppose that $s\in RO(P)\cap M$ is a condition weaker than $r$. Let $A=\{t\in P:t\leq s\}\in M$ and argue that $A$ is $r\cap M$-large.
This will conclude the proof since $P$ is dense in $RO(P)$ and so $\sum A=s$ and $s\in F$. Suppose for contradiction that $A$ is not $r\cap M$-large. Since the game $G(A, r\cap M)$ is determined, there must be a winning strategy $\gs\in M$ for Player I in it. Now, let $\langle M_k, x_k, a_k\rangle:k\in l$ enumerate $r\setminus M$ in $\in$-increasing order and consider the counterplay of Player II
against the strategy $\gs$ in which Player II's moves are $x_k$ for $k\in l$ in this order. Note that the strategy $\gs$ belongs to all
models $M_k$ for $k\in l$ and so this is a legal counterplay. At the end of it, Player II is in a winning position, as witnessed by the condition $r\in A$. This contradicts the assumption that $\gs$ was a winning strategy for Player I.
\end{proof}

It is natural to ask which traditional fusion-type forcings are Y-proper. We do not have a comprehensive answer to this question.
Instead, we prove a rather limited characterization theorem which nevertheless illustrates the complexity of the question well. 
For an ideal $I$ on $\gw$ let $P(I)$ be the poset of all trees $T\subset\gwtree$ which
have a trunk $t$ and for every node $s\in T$ extending the trunk, the set $\{n\in\gw:s^\smallfrown n\in T\}$ does not belong to $I$. The ordering is that of inclusion. Rather standard fusion arguments show that posets of this form are all proper, and they preserve $\gw_1$ covers on compact Polish spaces consisting of $G_\gd$ sets.

\begin{theorem}
\label{firsttheorem}
Let $I$ be an ideal on $\gw$. If $I$ is the intersection of $F_\gs$-ideals, then $P(I)$ is Y-proper.
\end{theorem}

In particular, Laver forcing is Y-proper. The implication in Theorem~\ref{firsttheorem} cannot be reversed already for $\mathbf{\gS}^0_4$ ideals. The ideal $I$ on $\gw\times\gw$ generated by vertical sections and sets with all vertical sections finite is not the intersection of $F_\gs$-ideals, but the poset $P(I)$ is Y-proper. The simplest example in which we do not know how to check the status of Y-properness is $I=$the ideal of nowhere dense subsets on $\bintree$.

\begin{theorem}
\label{secondtheorem}
Let $I$ be an analytic P-ideal on $\gw$. The following are equivalent:

\begin{enumerate}
\item $I$ is the intersection of $F_\gs$-ideals;
\item $P(I)$ is Y-proper;
\item for every compact Polish space $X$ and every open set $H\subset X^\gw$, every $H$-anticlique in the $P(I)$-extension is covered by countably many
$H$-anticliques in the ground model.
\end{enumerate}
\end{theorem}

An analytic P-ideal is an intersection of $F_\gs$-ideals if and only if it is the intersection of countably many $F_\gs$-ideals. A good example of an analytic P-ideal which can be written as such an intersection and yet is not $F_\gs$ is $I=\{a\subset\gw\colon \forall \eps>0\ \sum_{n\in a} n^{-\eps}<\infty\}$. An example of an analytic P-ideal which is not an intersection of $F_\gs$-ideals is the ideal of sets of asymptotic density zero; in fact, the only $F_\gs$ ideal containing the density ideal is trivial, containing $\gw$ as an element.

\begin{proof}[Proof of Theorem~\ref{firsttheorem}]
 We will need a bit of notation. Write $P=P(I)$.  Let $T\in P$ be a tree with trunk $t$. For a function $f\colon\gwtree\to I$ write $T_f=\{s\in T\colon \forall i\in\dom(s\setminus t)\ s(i)\notin f(s\restriction i)\}$, observe that the trees
$\{T_f\colon f\in \gw^{(\gwtree)}\}$ form a centered system in $P$, and pick an ultrafilter $F(T)\subset RO(P)$ extending this centered system.

\begin{claim}
\label{directclaim}
For every element $p\in F(T)$ there is a tree $S\subset T$ with trunk $t$ such that $S\leq p$.
\end{claim}

\begin{proof}
Suppose this fails for some $p$. Let $U$ be the set of all nodes $s\in T$ such that $t\subseteq s$ and there is no tree $S\subset T$
with trunk $s$ such that $S\leq p$. Observe that $t\in U$, and if $s\in U$ then the set $\{i\in\gw:
s^\smallfrown i\in T$ and $s^\smallfrown i\notin U\}$ is in $I$. If this failed, then there would be a $I$-positive set $a\subset\gw$
such that for each $i\in a$, $s^\smallfrown i\in T$ and there is a tree $S_i\subset T\restriction s^\smallfrown i$ with trunk $s^\smallfrown i$ which is below $p$. Now, $S=\bigcup_{i\in a}S_i=\sum_{i\in a}S_i\leq p$, contradicting the assumption that $s\in U$.

Now, let $V=\{s\in T:$ if $\dom(t)\leq i\leq\dom(s)$ then $s\restriction i\in U\}$ and use the previous paragraph to see
that $V=T_f$ for some $f:\gwtree\to I$. Since $p\in F(T)$, $p$ is compatible with $V$ and there is some tree $W\subset V$ such that $W\leq p$. Let $s$ be the trunk of $W$ and obtain a contradiction with the fact that $s\in U$.
\end{proof}

The (ultra)filters on $RO(P)$ critical for Y-properness of $P$ will be obtained in the following way. If $T\in P$ is a tree with trunk $t$,
write $a=\{i\in\gw\colon t^\smallfrown i\in T\}\notin I$, use the assumption on the ideal $I$ to find an $F_\gs$-ideal $I(T)$ such that $I\subset I(T)$ and $a\notin I(T)$, and an ultrafilter $U(T)$ on $\gw$ such that $a\in U(T)$ and $I(T)\cap U(T)=0$. Finally, let
$G(T)=\{p\in RO(P)\colon\{i\in a:p\in F(T\restriction t^\smallfrown i)\}\in U(T)\}$. It is not difficult to see that $G(T)$ is an ultrafilter.

Now we are ready for the fusion argument. Let $M\prec H_\theta$ be a countable elementary submodel containing $U, G, F$.

\begin{claim}
\label{ceclaim}
If $T\in M\cap P$ is a tree with trunk $t$, then there is a tree $S\subset T$ with the same trunk such that

\begin{enumerate}
\item for all $i\in\gw$, if $t^\smallfrown i\in S$ then $S\restriction t^\smallfrown i\in M$;
\item for every $p\in G(T)\cap M$, for all but finitely many $i$, if $t^\smallfrown i\in S$ then $S\restriction t^\smallfrown i\leq p$.
\end{enumerate} 
\end{claim}

\begin{proof}
Let $\{p_i\colon i\in\gw\}$ be an enumeration of $G(T)\cap M$. Let $\mu$ be a lower semicontinuous submeasure on $\gw$ such that $I(T)=\{a\colon \mu(a)<\infty\}$. By induction on $j\in\gw$ find finite pairwise disjoint sets $a_j\subset\gw$ and trees $S_k\subset T$ for each $k\in a_j$
so that

\begin{itemize}
\item $\mu(a_j)\geq j$;
\item $S_k\in M$ is a tree with trunk $t^\smallfrown k$, below $\bigwedge_{i<j}p_i$.
\end{itemize}

\noindent This is easy to do using Claim~\ref{directclaim} and elementarity of the model $M$ repeatedly. In the end, let $S=\bigcup_jS_j$.
\end{proof}

Now, an obvious fusion argument using Claim~\ref{ceclaim} repeatedly gives the following. For every tree $T\in M\cap P$
there is a tree $S\subset T$ in $P$ with the same trunk such that for every node $s\in S$ there is an ultrafilter $F(s)\in M$
on $RO(P)$ such that for every $p\in F(s)\cap M$ for all but finitely many $i\in\gw$, either $s^\smallfrown i\notin S$ or $S\restriction s^\smallfrown i\leq p$. We will verify that the condition $S\leq T$ is Y-master for the model $M$.

Indeed, let $U\leq S$ be any tree, and let $s$ be its trunk. We claim that for every $p\in RO(P)\cap M$, if $p\geq U$ then $p\in F(s)$. Indeed, if this failed then $1-p\in F(s)$, by the properties of the tree $S$ one can erase finitely many immediate successors
of $s$ in the tree $U$ to get some $V\subset U$ such that $V\leq 1-p$, and then $V$ would be a common lower bound of
$p$ and $1-p$, a contradiction.
\end{proof}

\begin{proof}[Proof of Theorem~\ref{secondtheorem}]
It is enough to show that (3) implies (1).
Suppose that $I$ is an analytic P-ideal which is not an intersection of $F_\gs$-ideals; we must show that there is a condition
in $P=P(I)$ forcing an anticlique which is not covered by countably many ground model anticliques.
Let $a\subset\gw$ be a
set which belongs to every $F_\gs$-ideal containing $I$, yet $a\notin I$. To simplify the notation, assume that $a=\gw$; otherwise, work under the condition $a^{<\gw}\in P$.

Let $M\prec H_\theta$ be a countable elementary submodel containing $I$. Let $Y$ be the compact Polish space of all ultrafilters on $RO(P)\cap M$. Let $X=K(Y)$ and consider the open set $H\subset X^\gw$ consisting of all sequences
$\langle K_n:n\in\gw\rangle\in X^\gw$ such that $\bigcap_nK_n=0$. By compactness, the set $H$ is open. For every condition $T\in P$, let $K_T=\{F\in Y\colon \{p\in RO(P)\cap M\colon p\geq T\}\subset F\}$. This is a compact subset of $Y$, therefore an element of $X=K(Y)$. Let $\dot A=\{K_T\colon T$ is a tree in the $P(I)$-generic filter$\}$. Clearly, this is a $P$-name for an $H$-anticlique. We will show that $\dot A$ is forced not to be covered by countably many $H$-anticliques in the ground model.

Suppose that $\{B_i\colon i\in\gw\}$ is a countable collection of $H$-anticliques and $T\in P$ is a condition. We will find a condition $S\leq T$ such that $K_S\notin \bigcup_iB_i$; this will complete the proof. By compactness, for every $i\in\gw$ there is an ultrafilter $F_i$ on $RO(P)\cap M$ such that $F_i\in\bigcap B_i$. We will find the condition $S\leq T$ so that for every $i\in\gw$
there is $p_i\in F_i$ such that $1-p_i\geq S$. Then, for every $i\in\gw$ $F_i\notin K_S$ and therefore $K_S\notin B_i$
as required.

The construction of the condition $S$ starts with a small claim:

\begin{claim}
\label{littleclaim}
For every tree $U\in P$ with trunk $t$ and every $j\in\gw$ there is a tree $V\leq U$ with the same trunk such that

\begin{enumerate}
\item there is a set $c\subset\gw$ such that $V=\{s\in U\colon s$ is compatible with $t^\smallfrown n\}$ for every $n\in c$;
\item for every $i\in j$ there is a condition $q_i\in F_i$ such that $V\leq 1-q_i$;
\item for every $i\in\gw$ there is a finite set $u$ of immediate successors of $t$ in the tree $V$ and a condition $q_i\in F_i$
such that the tree $V$ with the nodes in $u$ erased is below $1-q_i$.
\end{enumerate}
\end{claim}

\begin{proof}
Finally, we will use the assumptions on the ideal $I$. By a result of Solecki~\cite{solecki:ideals}, since $I$ is an analytic P-ideal it is possible
to find a lower semicontinuous submeasure $\mu$ on $\gw$ such that $I=\{b\subset\gw:\limsup_n\mu(b\setminus n)=0\}$.
Observe that for every $k\in\gw$ there is a partition of $\gw$ into finitely many singletons and finitely many pieces of $\mu$-mass $<2^{-k}$. If this failed, then the singletons together with sets of $\mu$-mass $<2^{-k}$ generate an $F_\gs$-ideal which contains $I$ as a subset and does not contain $\gw$ as an element, contradicting our assumptions on $I$. By the elementarity of the model $M$,
such partitions exist in the model $M$ as well.

Let $a=\{n\in\gw\colon t^\smallfrown n\in U\}$. Let $\eps=\limsup_n\mu(a\setminus n)>0$. Let $\langle k_i\colon i\in\gw\rangle$ be a sequence of numbers such that $\sum_i{2^{-k_i}}<\eps/2$. The previous paragraph shows that there are sets $b_i\subset\gw$
in the model $M$ such that each $b_i$ is either a singleton or a set of $\mu$-mass $<2^{-k_i}$ such that 
either the Boolean value $q_i=\|$the generic element of $\baire$ does not start with $t\|$ is in $F_i$, or the Boolean value
$q_i=\|$the generic element of $\baire$ starts with $t^\smallfrown n$
for some $n\in b_i\|$ is in the ultrafilter $F_i$.
It is now easy to find a set $c\subset a$ such that $\limsup_n\mu(c\setminus n)>\eps/2$
such that for all $i\in j$, $b_i\cap c=0$, and for every $i\in\gw$ $b_i\cap c$ is finite. The tree $V=\{s\in U:s$ is compatible with some $t^\smallfrown n$ for some $n\in c\}$ clearly works as desired.
\end{proof}

Assume for simplicity that the trunk of the tree $T$ is empty. A standard fusion argument using Claim~\ref{littleclaim} repeatedly yields a tree $S\leq T$ with empty trunk such that for every $i\in\gw$, there is a nonempty finite tree $u_i\subset S$ such that for every node $t\in u_i$
there is an element $q_i^t\in F_i$ such that the tree $S_t$ obtained from $S$ by restricting to $t$ and erasing all immediate successors of $t$ which are in $u_i$, is stronger than $1-q_i^t$. Let $p_i=\prod_{t\in u_i}q_i^t$ and observe that $S, p_i$ work as desired.
\end{proof}

\begin{question}
Does the conjunction of Y-properness and c.c.c.\ imply Y-c.c.?
\end{question}

\begin{question}
Suppose that $I$ is a suitably definable $\gs$-ideal on a Polish space $X$. Suppose that the quotient poset $P_I$ of Borel $I$-positive sets ordered by inclusion is proper. Are the following equivalent?

\begin{enumerate}
\item $P_I$ is Y-proper;
\item for every Polish compact space $Y$ and every open set $H\subset Y^\gw$, every $H$-anticlique in the $P_I$ extension is covered by countably many $H$-anticliques in the ground mode.
\end{enumerate}
\end{question}

\section{General treatment}
\label{generalsection}

Y-c.c.\ and Y-properness are preserved under a suitable notion of iteration, and there are suitable forcing axioms associated with them. The treatment is complicated enough to warrant a more general approach of which Y-c.c.\ and Y-properness are the most important instances.

\begin{definition}
\label{regularitydefinition}
A property $\Phi(F, B)$ of subsets $F$ of complete Boolean algebras $B$ is a \emph{regularity property} if the following is provable in ZFC:

\begin{enumerate}
\item (nontriviality) $\Phi(\{1\}, B)$ for every complete Boolean algebra $B$;
\item (closure up) $\Phi(F, B)\to\Phi(F', B)$ whenever $F'=\{p\in B:\exists q\in F\ q\leq p\}$;
\item (restriction) whenever $p\in B$ then $\Phi(F, B)$ implies $\Phi(F\cap (B\restriction p), B\restriction p)$, and $\Phi(F, B\restriction p)$ implies $\Phi(F, B)$. Here, $B\restriction p$ is the Boolean algebra $\{q\in B:q\leq p\}$ with the usual operations;
\item (complete subalgebras) if $B_0$ is a complete subalgebra of  $B_1$: for every $F\subset B_1$ $\Phi(F, B_1)\to \Phi(F\cap B_0, B_0)$ holds, and for every $F\subset B_0$ $\Phi(F, B_0)\to\Phi(F, B_1)$ holds;
\item (iteration) if $\dot B_1$ is a $B_0$-name for a complete Boolean algebra, $F_0\subset B_0$, $\dot F_1$ a name for a subset of $B_1$,
$\Phi(F_0, B_0)$ and $1\Vdash\Phi(\dot F_1, \dot B_1)$, then $\Phi(F_0*\dot F_1, B_0*\dot B_1)$ where 

$$F_0*\dot F_1=\{\langle p_0, \dot p_1\rangle\in B_0*\dot B_1:
p_0\land \|\dot p_1\in\dot F_1\|\in F_0\}$$.
\end{enumerate}

\noindent If the Boolean algebra $B$ is clear from the context, we write $\Phi(F)$ for $\Phi(F, B)$.
\end{definition}

In the last item, we use the Boolean presentation of the two-step iteration. Let $B_0$ be a complete Boolean algebra and $\dot B_1$ a $B_0$-name for a complete Boolean algebra. Consider the poset of all pairs $\langle p_0, \dot p_1\rangle$ such that
$p_0\in B_0$, $\dot p_1$ is a $B_0$-name for an element of $\dot B_1$, $p_0\neq 0$ and $p_0\Vdash\dot p_1\neq 0$. The ordering is defined
by $\langle q_0, q_1\rangle\leq \langle p_0, \dot p_1\rangle$ if $q_0\leq p_0$ and $q_0\Vdash\dot q_1\leq\dot p_1$. It is not difficult to check that the separative quotient of this partial ordering (together with a zero element) is complete (admits arbitrary suprema and infima) and therefore forms a complete Boolean algebra which we will denote by $B_0*\dot B_1$.

The central example of a regularity property studied in this paper is $\Phi(F)=$``$F$ is a centered set.'' Other possibilities include $\Phi(F)=$``any two elements of $F$ are compatible'' or $\Phi(F)=$``for every collection $\{p_n:n\in\gw\}\subset F$ the Boolean value $\liminf_n p_n$ is nonzero.'' There are many other sensible possibilities.
Items (1--3) of the definition imply that the strongest conceivable regularity property is $\Phi(F)=$``$F$ has a lower bound.''

The class of regularity properties is closed under countable conjunctions. The disjunctions are more slippery but also more rewarding.
To treat them, we introduce an additional notion.

\begin{definition}
Let $G$ be a set with a binary operation $*$.
A property $\Phi(g, F, B)$ of subsets $F$ of complete Boolean algebras $B$ and elements $g\in G$ is a $G$-\emph{regularity property}
if for each $g\in G$, $\Phi(g, \cdot, \cdot)$ satisfies the demands (1--4) of Definition~\ref{regularitydefinition}
and (5) is replaced with

\begin{enumerate}
\item[(5)] if $\dot B_1$ is a $B_0$-name for a complete Boolean algebra, $F_0\subset B_0$, $\dot F_1$ a name for a subset of $B_1$,
$\Phi(g_0, F_0, B_0)$ and $1\Vdash\Phi(\check g_1, \dot F_1, \dot B_1)$, then $\Phi(g_0*g_1, F_0*\dot F_1, B_0*\dot B_1)$.
\end{enumerate}
\end{definition}

\noindent A typical case appears when $G$ is a countable semigroup. If $G$ is clear from context, we omit it from the notation. It is clear that every regularity property is a $G$-regularity property for $G=\{1\}$ with the multiplication operation. Good nontrivial examples include $G=$the rationals in the interval $(0,1]$ with multiplication, and $\Phi(\eps, F, B)=$``there is a finitely additive probability measure $\mu$ on $B$ such that $\mu(p)\geq\eps$ for all $p\in F$." Another example studied by Steprans \cite[Definition~3]{steprans:strong-Q} is obtained when $G\subset\baire$ is a set closed under composition, with the composition operation, and $\Phi(g, F, B)=$``for every $n\in\gw$ and every collection of $g(n)$ many elements of $F$, there are $n$ many elements in the collection with a common lower bound."

\begin{definition}
Suppose that $\langle G, *\rangle$ is a set with a binary operation.
Suppose that $\Phi$ is a $G$-regularity property of subsets of complete Boolean algebras. 

\begin{enumerate}
\item  A poset $P$ is $\Phi$\emph{-c.c.} if for every condition $q\in P$ and every countable elementary submodel $M\prec H_\theta$ containing $P, G$
there is an element $g\in G\cap M$ and a set $F\in M$ such that $\Phi(g, F)$ holds and $F$ contains all elements of $RO(P)\cap M$ weaker than $q$.
\item $P$ is $\Phi$\emph{-proper} if for every countable elementary submodel $M\prec H_\theta$ containing $P, G$ and every condition $p\in P\cap M$
there is a $\Phi$-\emph{master condition} $q\leq p$: this is a condition which is master for $M$ and for every $r\leq q$, there is an element $g\in G\cap M$ and a set $F\in M$ such that $\Phi(g, F)$ holds and $F$ contains all elements of $RO(P)\cap M$ weaker than $q$.
\end{enumerate}
\end{definition}

Clearly, Y-c.c.\ and Y-properness are special cases of $\Phi$-c.c.\ and $\Phi$-properness where $\Phi(F)=$``$F$ is a centered set.''
Certain natural posets may satisfy other variations of $\Phi$-c.c.\ For example, the random poset satisfies $\Phi$-c.c. for $\Phi(F)=$``any two elements of $F$ are compatible'' or $\Phi(F)=$``for every collection $\{p_n:n\in\gw\}\subset F$ the Boolean value $\liminf_n p_n$ is nonzero.''
The notion of strong properness (as defined in~\cite{mitchell:addclub}) corresponds to $\Phi$-properness for $\Phi(F)=$``$F$ has a lower bound.'' Every strongly proper poset is thus $\Phi$-proper for every choice of the regularity property~$\Phi$.

There are many attractive arguments drawing abstract consequences from $\Phi$-c.c.\ and $\Phi$-properness for various regularity properties $\Phi$. We will limit ourselves to several striking consequences of this kind.

\begin{theorem}
Suppose that $\Phi$ is a regularity property such that $\Phi(F)$ implies that $F$ contains no uncountable antichain. Then $\Phi$-c.c.\ implies c.c.c.\ 
\end{theorem}

\begin{proof}
For contradiction, assume that $P$ is a $\Phi$-c.c.\ poset with an antichain $A$ of size $\aleph_1$. Let $M\prec H_\theta$ be a countable elementary submodel containing $P, A$, and let $q\in A\setminus M$ be any element. Let $F\in M$
be a subset of $RO(P)$ such that $\Phi(F)$ holds and $F$ contains all elements of $RO(P)\cap M$ weaker than $r$.
Let $I$ be the $\gs$-ideal on $A$ $\gs$-generated by the sets $B\subset A$ such that $\sum B\notin F$.

\begin{claim}
$I$ is a nontrivial c.c.c.\ $\gs$-ideal containing all singletons.
\end{claim}

\begin{proof}
For the nontriviality, use the elementarity of the model $M$. If $B_n\subset A$ for $n\in\gw$ are generating elements of the $\gs$-ideal $I$ in the model $M$, then $q\notin B_n$ for each $n$ by the definitions, and so $q\notin \bigcup_nB_n$ and $\bigcup_nB_n\neq A$. Thus, no countable union of generating sets in the model $M$ of the $\gs$-ideal $I$ covers all of $A$,
and by the elementarity of the model $M$ this is true even for generating sets in $V$. 

If the $\gs$-ideal $I$ failed to be c.c.c.\ then there would be an uncountable collection $C$ of pairwise disjoint $I$-positive sets.
As the sets in $C$ are pairwise disjoint and $A$ is an antichain, the Boolean sums $\sum B$ for $B\in C$ are pairwise incompatible. They all must be elements of $F$ by the definition of $I$. However, this contradicts the assumption on the regularity property $\Phi$. 
\end{proof}

However, by a classical theorem of Ulam \cite{ulam:matrix}, in ZFC there are no nontrivial c.c.c.\ $\gs$-ideals on sets of size $\aleph_1$ which contain no singletons. This is a contradiction.
\end{proof}

\begin{theorem}
\label{litheorem}
Let $\Phi$ be a regularity property such that $\Phi(F)$ implies that $F$ contains no infinite antichain. For every $\Phi$-proper
poset $P$, if $H\subset [X]^2$ is an open graph on a second countable space $X$, then every $H$-anticlique in the $P$-extension is covered by countably many anticliques in the ground model.
\end{theorem}

\noindent Note that the statement ``$F$ contains no infinite antichains'' in itself is not a regularity property as it does not satisfy the iteration clause of regularity.

\begin{proof}
Suppose that $P$ is a $\Phi$-proper poset and $H\subset [X]^2$ is an open graph on a second countable space. Let $\dot A$ be a $P$-name for an anticlique and let $F\subset RO(P)$ be a set satisfying $\Phi$.

\begin{claim}
The set $B(\dot A, F)=\{x\in X:$ for every open neighborhood $O\subset X$ of $x$, the Boolean value $\|\check O\cap\dot A\neq 0\|$ is in $F\}$
is a union of countably many $H$-anticliques.
\end{claim}

\begin{proof}
Remove all basic open neighborhoods $O$ from the set $B(\dot A, F)$ such that $O\cap B(\dot A, F)$ is a union of countably many $H$-anticliques; it will be enough to show that the remainder $B$ is empty. Suppose not; then for every basic open set $O\subset X$, the set $B\cap O$, if nonempty, is not an $H$-anticlique. This allows us to build by induction on $n\in\gw$ basic open sets $O_n, U_n\subset X$ such that

\begin{itemize}
\item $O_n\times U_n\subset H$;
\item $O_{n+1}, U_{n+1}\subset U_n$;
\item the sets $B\cap O_n$ and $B\cap U_n$ are both nonempty.
\end{itemize}

For each $n\in\gw$, let $p_n\in F$ be the Boolean value of $\|\check O_n\cap\dot A\neq 0\|$. By the assumption on $\Phi$,
there must be numbers $n\neq m$ such that the conditions $p_n, p_m$ are compatible. Denote their lower bound by $q$.
Then $q\Vdash\dot A\cap\check O_n\neq 0$ and $\dot A\cap\check O_m\neq 0$, which together with the fact that $O_n\times O_m\subset H$
contradicts the assumption that $\dot A$ is forced to be an $H$-anticlique.
\end{proof}

Now, let $p\in P$ be a condition, let $M\prec H_\theta$ be a countable elementary submodel containing $P, p,\dot A, H, X$. Let 
$q\leq p$ be a $\Phi$-master condition for the model $M$. We claim that $q$ forces $\dot A$ to be covered by the $H$-anticliques in the model $M$; this will complete the proof.

Suppose that this fails and let $r\leq q$ and $x\in X$ be a point which is not in any anticlique in the model $M$ and yet $r\Vdash\check x\in\dot A$. Let $F\in M$ be a set satisfying $\Phi$ and containing all conditions $s\in RO(P)\cap M$ such that $s\geq r$. Then,
for every basic open set $O\subset X$ containing $x$ it is the case that $\|\check O\cap\dot A\neq 0\|\geq r$, and since the Boolean value
is an element of the model $M$, it is the case that $\|\check O\cap\dot A\neq 0\|\in F$ and so $x\in B(F, \dot A)$. The latter set is a union of $H$-anticliques in the model $M$ as per the claim. This is a contradiction.
\end{proof}

Steprans \cite{steprans:cellularity}and Todorcevic \cite[Theorem 7]{todorcevic:cellularity} produced for every number $k\geq 2$ a poset $P_k$ which is $\gs$-$k$-linked and yet adds an anticlique for an open hypergraph  in dimension $k+1$ which is not covered by countably many anticliques in the ground model. Thus, the various finite dimensions of open hypergraphs do have significance. Once finitely additive measures enter the picture, all finite dimensions are well-behaved:

\begin{theorem}
Suppose that $\Phi$ is a regularity property such that $\Phi(F)$ implies that there is a finitely additive probability measure $\mu$ on $B$ and a real number $\eps>0$ such that $\forall p\in F\ \mu(p)>\eps$.
Then, for every $n\in\gw$, every second countable space $X$, and every open set $H\subset X^n$, every $H$-anticlique in $\Phi$-proper extension is covered by countably many ground model $H$-anticliques. 
\end{theorem}

\begin{proof}
Let $P$ be a $\Phi$-proper poset and $\dot A$ a $P$-name for an $H$-anticlique. Let $F\subset RO(P)$ be a set with $\Phi(F)$. Let $B(\dot A, F)=\{x\in X:$ for every open neighborhood $O\subset X$ with $x\in O$,
$\|\check O\cap\dot A\neq 0\|\in F\}$.

\begin{claim}
$B(\dot A, F)\subset X$ is a union of countably many $H$-anticliques.
\end{claim}

\begin{proof}
First, remove from the set $B(\dot A, F)$ all open neighborhoods in which the set is the union of countably many anticliques. We claim that the remainder $B\subset X$ is empty; this will complete the proof of the claim.

Suppose for contradiction that $B\neq 0$. Note that for every open neighborhood $O\subset X$, if $O\cap B\neq 0$ then $O\cap B$ is not an $H$-anticlique. Let $\mu$ be a finitely additive probability measure on $RO(P)$
such that for some fixed $\eps>0$, $\mu(p)\geq\eps$ for every condition $p\in F$. For every open set $O\subset X$, write $q(O)=\|\check O\cap\dot A\neq 0\|$.
By induction on $m\in\gw$
build basic open sets $O_m^i:i\in n$ and numbers $0\neq i_m\in n$ so that

\begin{itemize}
\item for every $i\in n$ it is the case that $B\cap O_m^i\neq 0$;
\item $\prod_iO_m^i\subset H$;
\item $O_{m+1}^i\subset O_m^{i_m}$;
\item writing $q_m=q(O_m^0)-q(O_m^{i_m})$, it is the case that $\mu(q_m)\geq\eps/n$.
\end{itemize}

\noindent This is easy to do: at stage $m$, the set $B\cap O_m^{i_m}$ is nonempty and therefore not an anticlique, which makes it possible to find sets $O_{m+1}^i$ for $i\in n$ satisfying the first three items.
Now, since $\mu(q(O_m^0))>\eps$, if for every number $0\neq i\in n$ it were the case that $\mu(q(O_m^0)-q(O_m^i))<\eps/n$, then the conjunction
$\bigwedge_iq(O_m^i)$ would have positive $\mu$-mass by the finite additivity of $\mu$. This conjunction would force $\dot A$ to contain an $H$-edge, contradicting the initial assumptions.

In the end, the conditions $q_m$ for $m\neq 0$ form an antichain and each of them has $\mu$-mass at least $\eps/n$, a contradiction with the finite additivity of the probability measure $\mu$.
\end{proof}

The rest of the argument follows word by word the conclusion of the proof of Theorem~\ref{litheorem}.
\end{proof}

\begin{theorem}
Suppose that $\Phi$ is a regularity property such that $\Phi(F)$ implies that $F$ contains no infinite antichains. Suppose that $P$ is a $\Phi$-proper poset and $\kappa$ is a cardinal. For every function $f\in\kappa^\kappa$ in the $P$-extension, if $f\restriction a$ is in the ground model for every countable ground model set $a\subset\kappa$, then $f$ is in the ground model.
\end{theorem}

\begin{proof}
We will start with an abstract claim. Let $\kappa$ be an uncountable cardinal. A \emph{coherent system} on $\kappa$ is a collection $S$ of partial countable functions on $\kappa$, closed under subsets, such that for every countable set $a\subset\kappa$ there is $g\in S$ with $\dom(g)=a$, and there is no infinite collection of pairwise incompatible functions in $S$.

\begin{claim}
For every coherent system $S$ on $\kappa$, the set $H=\{f\in\kappa^\kappa$: for every countable set $a\subset \kappa$, $f\restriction a\in S\}$ is nonempty and finite.
\end{claim}

\begin{proof}
To see that the set $H$ is nonempty, consider the sets $H_a=\{f\in\kappa^\kappa\colon f\restriction a\in S\}$ for every countable set $a\subset\kappa$. Intersection of any countable collection of such sets is nonempty by the assumptions on $S$. Let $U$ be an ultrafilter on $\kappa^\kappa$ containing all sets $H_a$ for $a\subset\kappa$ countable. For each such set $a\subset\gw$,
there are only finitely many functions in $S$ with domain $a$, and so one of them, denoted by $g_a$, satisfies $\{f\in\kappa^\kappa\colon g_a\subset f\}\in U$. It is immediate that $\bigcup_ag_a\in H$.

To prove the finiteness of $H$, suppose for contradiction that $f_n$ for $n\in\gw$ are pairwise distinct functions in $H$. Then, there is a countable set $a\subset\kappa$ such that the functions $f_n\restriction a$  for $n\in\gw$ are pairwise distinct. They all belong to the set $S$, contradicting the coherence assumption on $S$. 
\end{proof}

Now suppose that $P$ is a $\Phi$-proper poset and $\dot f$ is a $P$-name for a function from $\kappa$ to $\kappa$.
Let $p\in P$ be a condition forcing $\dot f\restriction a\in V$ for every countable set $a\subset\kappa$; we must produce a function $e\in\kappa^\kappa$ and a stronger condition forcing $\check e=\dot f$. Let $M\prec H_\theta$ be a countable elementary submodel containing $P, p, \dot f$, and let $q\leq p$ be a $\Phi$-master condition for $M$. Find a condition $r\leq q$ deciding all values of $\dot f\restriction M$, yielding a function $h\colon M\to\kappa$. We will find a function $e\in M\cap \kappa^\kappa$ such that
$h\subset e$. Then, since $r$ is a master condition for $M$ and $r\Vdash\check e\restriction M=\dot f\restriction M(=\check h)$, it must be the case that $r\Vdash \check e=\dot f$. This will complete the proof.

Towards the construction of the function $e$, let $F\subset RO(P)$ be an upwards closed set in the model $M$ such that $\Phi(F)$ holds and $F$ contains all elements of $RO(P)\cap M$ weaker than $r$.  Let $S=\{g\colon g$ is a partial function from $\kappa$ to $\kappa$ with countable domain and the Boolean value $\|\check g\subset\dot f\|$ belongs to $F\}$. We claim that $S\in M$ is a coherent system. Closure of $S$ under subsets is clear from the definitions. $S$ contains no infinite set of pairwise incompatible functions since $F$ contains no infinite antichain. For every countable set $a\in M$ the function $h\restriction a$ is in $M\cap S$ since $r\Vdash \dot f\restriction a=h\restriction a\in V$ and $r$ is a master condition for the model $M$. By the elementarity of the model $M$, the coherence of the system $S$ follows.

Now, let $H\in M$ be the finite set of functions from $\kappa$ to $\kappa$ obtained by the application of the claim to the coherent system $S$. We claim that the function $h$ is a subset of one element of $H$. Indeed, if this was not the case, then there would be a finite set $c\subset\kappa\cap M$ such that $h\restriction c$ is not a subset of any function in the finite set $H$.
Let $T=\{g\in S\colon g$ is a function compatible with $h\restriction c\}$. Just as in the previous paragraph, $T\in M$ is a coherent system,
and there is a function $e\in\kappa^\kappa$ such that every restriction of $e$ to a countable set is in $T$. This function must appear on the finite list $H$ while $e\restriction c=h\restriction c$. Contradiction!
\end{proof}

\section{Iteration theorems}

As with most forcing properties, the point of the properties introduced in the previous section is that they are preserved under suitable iterations and their associated forcing axioms can be forced with a poset in the same category. 

\begin{definition}
Suppose that $\Phi$ is a $G$-regularity property of subsets of complete Boolean algebras. 

\begin{enumerate}
\item if $\kappa$ is a cardinal then $\Phi$-$\mathrm{MA}_\kappa$ is the statement that for every c.c.c.\ $\Phi$-c.c.\ poset $P$ and every list of open dense subsets of $P$ of size $\kappa$ there is a filter on $P$ meeting them all;
\item $\Phi$-PFA is the statement that for every $\Phi$-proper poset $P$ and every list of $\aleph_1$ many open dense subsets of $P$ there is a filter on $P$ meeting them all.
\end{enumerate}
\end{definition}

In the important special case of $\Phi(F)=$``$F$ is centered'', we will write $\mathrm{YMA}_\kappa$ and YPFA for $\Phi$-$\mathrm{MA}_\kappa$ and $\Phi$-PFA.

\begin{theorem}
\label{fsitheorem}
Let $\Phi$ be a $G$-regularity property. Then the conjunction of c.c.c.\ and $\Phi$-c.c.\ is preserved under

\begin{enumerate}
\item restriction to a condition;
\item complete subalgebras;
\item the finite support iteration.
\end{enumerate}
\end{theorem}

\begin{proof}
The first two items follow easily from the subalgebra and restriction clauses of regularity. The two-step iteration part of (3) follows
just as easily from the iteration clause of regularity. If $P_0$ has $\Phi$-c.c.\ and $\dot P_1$ is a $P$-name such that $P_0\Vdash\dot P_1$ has $\Phi$-c.c., we must show that $P_0*\dot P_1$ has $\Phi$-c.c.

Let $M\prec H_\theta$ be a countable elementary submodel containing $P_0, \dot P_1$ and let $\langle q_0, \dot q_1\rangle$ be an arbitrary condition in the iteration.
We must find a set $F\in M$ in $RO(P_0)*RO(\dot P_1)$ and $g\in G\cap M$ such that $\Phi(g, F)$ holds, and for every condition $\langle p_0, \dot p_1\rangle\in RO(P_0)*RO(\dot P_1)$ in the model $M$, if $\langle p_0, \dot p_1\rangle\geq \langle q_0, \dot q_1\rangle$ then $\langle p_0, \dot p_1\rangle\in F$. To this end, write $\dot G_0$ for the canonical $P_0$-name for its generic filter and $M[\dot G_0]$ for the $P_0$-name for the set $\{\tau/\dot G_0:\tau\in M$ is a $P_0$-name$\}$. It is well known that $M[\dot G_0]$ is forced to be a countable elementary submodel of $H_\theta$ of the generic extension ${V[\dot G_0]}$ and its intersection with the ground model is equal to $M$.
Strengthening $q_0$ if necessary and using $\Phi$-c.c.\ of the poset $P_1$ in the extension, we may find a name $\dot F_1\in M$  for a subset of $RO(\dot P_1)$ and $g_1\in G\cap M$ such that $1\Vdash\Phi(\check g_1, \dot F_1)$, and 
$q_0\Vdash \{p\in RO(\dot P_1)\cap M[\dot G_0]\colon p\geq \dot q_1\}\subset\dot F_1$. Use the $\Phi$-c.c.\ of $P_0$ to find some $g_0\in G\cap M$ and $F_0\in M$ such that $F_0\subset RO(P_0)$, $\Phi(g_0, F_0)$ and $\{p\in RO(P_0)\cap M:p\geq q_0\}\subset F_0$.
By the iteration clause of regularity, $\Phi(g_0*g_1, F_0*\dot F_1)$ holds. We claim that $F=F_0*\dot F_1$ witnesses $\Phi$-c.c.\ for the iteration.

Indeed, suppose that $\langle p_0, \dot p_1\rangle\in M$ is a condition in the iteration weaker than $\langle q_0, \dot q_1\rangle$.
Thus, $q_0\Vdash\dot p_1\geq\dot q_1$, $\dot p_1\in M[\dot G_0]$, and so $\dot p_1\in\dot F_1$. The Boolean value $\|\dot p_1\in\dot F_1\|$ is in the model $M$ and it is weaker than $q_0$, so the conjunction $p_0\land \|\dot p_1\in\dot F_1\|\in M$ is still weaker than $q_0$ and so belongs to the set $F_0$. Thus, $\langle p_0, \dot p_1\rangle\in F_0*\dot F_1$ as desired.

The general proof proceeds by induction on $\gb=$the length of the iteration. The case $\gb$ successor is handled by the two-step iteration case. Suppose that
$\gb$ is limit, $M$ is a countable elementary submodel of $H_\theta$, and $q$ is any condition in the iteration. The domain of $q$ is a finite subset of $\gb$; let $\ga=\max(M\cap\dom(q))$. 
Write $P$ for the whole iteration, $P_0$ for the initial segment of the iteration up to $\ga$ inclusive, and $\dot P_1$
for the remainder of the iteration; thus, $\dot P_1$ is a $P_0$-name. The condition $q$ can be viewed as a pair $\langle q_0, \dot q_1\rangle$ where $q_0\in P_0$ and $q_0\Vdash\dot q_1\in\dot P_1$.
Since $\ga\in\gb$, the induction hypothesis guarantees the existence of a subset $F_0\in M$ of $RO(P_0)$ and an element $g\in M\cap G$ such that $\Phi(g, F_0)$ holds and for every condition $p\in RO(P_0)$ in the model $M$, weaker than
$q_0$, belongs to the set $F_0$. Let $F\in M$ be the subset of $RO(P)$  consisting of pairs $\langle p_0, \dot p_1\rangle\in RO(P_0)*RO(\dot P_1)$
where $p_0\land \|\dot p_1=1\|\in F_0$. By the nontriviality and the finite iteration clauses of regularity, $\Phi(g*h, F, RO(P))$ holds for every $h\in G$; we claim that the set $F$ works as desired.
Suppose that $p\geq q$ is a condition in the model $M$ in $RO(P)$; we must show that $p\in F$.
 The condition $p$ can be viewed as a pair
$\langle p_0, p_1\rangle$ such that $p_0\in RO(P_0)$ and $p_0\Vdash\dot p_1\in RO(\dot P_1)$. Since $p\geq q$,
it is the case that $p_0\geq q_0$ and $q_0\Vdash\dot p_1\geq \dot q_1$. 

The important point is that the latter formula means that $q_0\Vdash\dot p_1=1$.
 If this were not the case, by the c.c.c.\ of $P_1$ there would be a strengthening $r_0\leq q_0$ and a condition $\dot r_1\in M\cap P_1$ such that $r_0\Vdash \dot r_1$ is incompatible with $\dot p_1$. Now, by c.c.c.\ of $P_0$, $P_0$ forces the domain of $\dot r_1$ to be a subset of $M$ and therefore disjoint from $\dom(\dot q_1)$. Thus, $r_0\Vdash\dot r_1, \dot q_1$ are compatible, contradicting the assumption that
$q_0\Vdash\dot p_1\geq\dot q_1$.

Now, the Boolean value $\|\dot p_1=1\|\in RO(P_0)$ is an element of $M$ and it is weaker than $q_0$. The same is true of $p_0$. Therefore, the conjunction $p_0\land \|\dot p_1=1\|$ must belong to the set $F_0$, and so $\langle p_0, \dot p_1\rangle\in F$ as desired.
\end{proof}

As an abstract consequence of Theorem~\ref{fsitheorem}, it is possible to force Martin's Axiom for c.c.c.\ $\Phi$-c.c.\ posets
with a $\Phi$-c.c.\ poset. The following simple general theorem does not seem to appear in the literature:

\begin{theorem}
Suppose that $\Psi$ is a property of complete Boolean algebras which provably in ZFC implies c.c.c.\ and is preserved under complete subalgebras and the finite support iteration. Let $\kappa$ be an uncountable regular cardinal and suppose that $\Diamond_{\cof(\kappa)\cap\kappa^+}$ holds. There is a complete Boolean algebra satisfying $\Psi$ forcing $\mathrm{MA}_\kappa$ for posets satisfying $\Psi$.
\end{theorem}

\begin{corollary}
\label{ymatheorem}
Let $\Phi$ be a $G$-regularity property of sets of Boolean algebras. Let $\kappa$ be an uncountable regular cardinal and suppose that $\Diamond_{\cof(\kappa)\cap\kappa^+}$ holds. There is a c.c.c.\ $\Phi$-c.c.\ poset forcing $\Phi$-$\mathrm{MA}_\kappa$.
\end{corollary}

\begin{proof}
Let $\langle E_\ga:\ga\in\kappa^+\rangle$ be a diamond sequence for $\cof(\kappa)\cap\kappa^+$. This specifically means the following. Fix a wellordering $\prec$ of the set $H_{\kappa^+}$ of ordertype $\kappa^+$. Each set $E_\ga$ is of hereditary cardinality $\kappa$ and whenever $A\subset H_{\kappa^+}$ is a set, then the set 
$$\{\ga\in\cof(\kappa)\cap\kappa^+\colon E_\ga=\{x\in H_{\kappa^+}\colon x\in A\text{ and the rank of }x\text{ in }\prec\text{ is less than }\ga\}\}$$ 
is stationary.

In the following, for a poset $P$ we will write $\Psi(P)$ for the statement $\Psi(RO(P))$.
Consider the finite support iteration $R=\langle R_\ga, \dot Q_\ga\colon \ga\in\kappa^+\rangle$ obtained by the following rule: if $\ga$ is an ordinal such that $E_\ga$ codes an $R_\ga$-name for a poset, and in the $R_\ga$-extension $\Psi(\dot E_\ga)$ holds, then $\dot Q_\ga=E_\ga$. Otherwise, let $\dot Q_\ga=$the $R_\ga$-name for the trivial poset. We claim that the iteration works as required.

Suppose that in the $R$-extension, $P$ is a poset, $\Psi(P)$ holds, and $\langle D_\gb\colon \gb\in\kappa\rangle$ are open dense subsets
of it. We must produce a filter meeting them all. First of all, without loss of generality, we may assume that $\|P\|\leq\fc\leq\kappa^+$.
If this were not the case, let $N\prec H_\theta$ be an elementary submodel of size $\fc$ containing $P,\tau, \langle D_\gb:\gb\in\kappa\rangle$ as elements, $\kappa$ as a subset,
and such that $N^\gw\subset N$. Then, $N\cap P$ is a regular subposet of $P$, therefore $\Psi(N\cap P)$ holds by the closure of $\Psi$ under complete subalgebras, 
it has size $\leq\fc$ and all sets $D_\gb\cap N$ for $\gb\in\kappa$ are dense in it.
If there is a filter $G\subset N\cap P$ meeting all the sets $D_\gb\cap N$ for $\gb\in\kappa$, then we are done.

Thus, without loss of generality assume that $\|P\|=\kappa^+$, $P\subset H_{\kappa^+}$ and use the c.c.c. of $R$ find an $R$-name $\tau\subset H_{\kappa^+}$ for it so that $R\Vdash\Psi(\tau)$. Back in the ground model,
find an elementary submodel $N\prec H_\theta$ of size $\kappa$ containing $P,\tau, \langle D_\gb:\gb\in\kappa\rangle$ as elements, $\kappa$ as a subset,
such that $N^\gw\subset N$ and, writing $\ga= N\cap\kappa^+$, it is the case that $\tau\cap N=E_\ga$. We claim that $R_\ga\Vdash E_\ga$ is a poset satisfying $\Psi$,
thus $\dot Q_\ga=E_\ga$, and the generic filter added by the $\ga$-th stage of the iteration generates a filter on $P$
meeting all the dense subsets as required.

Let $G_\ga\subset R_\ga$ be a generic filter and for the remainder of the proof work in $V[G_\ga]$. Let $R^\ga$ be the remainder of the iteration,
so $\Psi(R^\ga)$ holds. Write $P_\ga=E_\ga/G_\ga$; thus $R^\ga\Vdash P_\ga\subset P$. The elementarity of the model $N$ has an important consequence:

\begin{claim}
The map $\pi\colon P_\ga\to R^\ga*\dot P$ given by $\pi(p)=\langle 1, \check p\rangle$ is a regular embedding.
\end{claim}

\begin{proof}
We must verify that if $A\subset P_\ga$ is a maximal antichain, then $\pi''A\subset R^\ga*\dot P$
is a maximal antichain as well, or equivalently $R^\ga\Vdash\dot A\subset\dot P$ is maximal.
To prove this, suppose that $G^\ga\subset R^\ga$ is a filter generic over the model $V[G_\ga]$, and let
$G\subset R$ be the concatenation of $G_\ga$ and $G^\ga$. Then in $V[G]$ the following holds:

\begin{itemize}
\item $N[G]$ is an elementary submodel of $H_\theta[G]$;
\item $P\cap N[G]=P_\ga$;
\item $(P_\ga)^\gw\cap V[G_\ga]\subset N[G]$.
\end{itemize}

\noindent For the last item, return to the ground model for a moment
and observe that every $R_\ga$-name $\gs$ for an element of $(\dot P_\ga)^\gw$ is at the same time an $R$-name for
an element of $(\dot P)^\gw$. At the same time, $N\cap R=R_\ga$ and $N$ is closed under countable sequences, therefore $N$ contains $\gs$ as an element.

It follows that
$A\in N[G]$. Since $A\subset P_\ga$ is a maximal antichain and $P\cap N[G]=P_\ga$, $N[G]\models A\subset P$ is a maximal antichain. Since $N[G]$ is elementary in $H_\theta[G]$, $A\subset P$ must be a maximal antichain as desired.
\end{proof}

Now, still arguing in the model $V[G_\ga]$, both steps in the iteration $R^\ga*\dot P$ satisfy $\Psi$ and so does the iteration.
$P_\ga$ is a regular subposet of this iteration and therefore satisfies $\Psi$ as well. Therefore, at stage $\ga$ of the iteration the poset $P_\ga$ is forced with,
and the resulting filter on $P_\ga\subset P$ meets all the open dense subsets on the list $\langle D_\gb:\gb\in\kappa\rangle$.
\end{proof}

Now, let us move to the proper variations.
$\Phi$-properness is not preserved under the countable support iteration. To provide a trivial example, consider the countable support iteration of an atomic poset with two atoms, of length $\gw_1$. Clearly, each poset in the iteration is Y-proper, and the iteration
is isomorphic to adding a subset of $\gw_1$ with countable approximations. This poset is not Y-proper by Theorem~\ref{yproperpreservationtheorem}(1). Even so, it is possible to force the forcing axiom for Y-proper posets with an Y-proper poset using the technology of \cite{neeman:twotypes}. This is the contents of the following theorem.

\begin{theorem}
\label{ypfaforcingtheorem}
Suppose that $\Phi$ is a $G$-regularity property and there is a supercompact cardinal. Then there is a $\Phi$-proper forcing $P$ forcing $\Phi$-PFA.
\end{theorem}

\begin{proof}
We first verify the preservation of $\Phi$-properness under two-step iteration. This follows immediately from the iteration clause of regularity:

\begin{claim}
\label{itclaim}
If $P$ is $\Phi$-proper and $P\Vdash\dot Q$ is $\Phi$-proper, then $P*\dot Q$ is $\Phi$-proper. If $M\prec H_\theta$ is a countable elementary submodel containing $P$, $\dot Q$, and $p\in P$ is a $\Phi$-master condition for $M$ in $P$ and $p\Vdash\dot q$ is a $\Phi$-master condition for $M[G]$ in $\dot Q$, then $\langle p, \dot q\rangle$ is a $\Phi$-master condition for $M$ in $P*\dot Q$.
\end{claim}

Let $\kappa$ be an inaccessible cardinal and $f:\kappa\to V_\kappa$ be a function. Let $I\subset\kappa+1$ be the set of all inaccessible cardinals $\gb$ such that $\langle V_\gb, f\restriction\gb\rangle\prec \langle V_\kappa, f\rangle$; in particular, $\kappa\in I$.

For every ordinal $\gb\in I$ define the orders $Q_\gb$ by $m\in Q_\gb$ if $m$ is a finite $\in$-chain whose elements are either countable elementary submodels of $V_\gb$ (the \emph{countable nodes})
or sets $V_\gd$ for some $\gd\in I\cap\gb$ (the \emph{transitive nodes}). Moreover, the chain $m$
must be closed under intersections. The ordering is that of reverse inclusion. Observe that if $\gd\in\gb$ are elements of $I$ then $Q_\gd\subset Q_\gb$.
It is proved in~\cite{neeman:twotypes} that $Q_\beta$ is strongly proper and hence also $\Phi$-proper.

By transfinite recursion on $\gb\in I$ we will define partial orders $\langle P_\gb, \leq_\gb\rangle$. The canonical names for their respective generic filters will be denoted by
$\dot G_\gb$. The elements of $P_\gb$ will be certain pairs $p=\langle m(p), w(p)\rangle$,
where $m(p)\in Q_\gb$ and $w(p)$ is a function on $m(p)$. For such a pair $p$, if $\gd\in I$ is such that $V_\gd\in m(p)$, write $p\restriction\gd$ for the pair $\langle m(p)\cap V_\gd, w(p)\restriction V_\gd\rangle$.
The posets $P_\gb$ are defined by the following recursive formula.

A set $p$ is an element of $P_\gb$ if $p=\langle m(p), w(p)\rangle$ where $m(p)\in Q_\gb$ and $w(p)$ is a function with $\dom(w(p))=m(p)$ such that
for every transitive node $V_\gd\in m(p)$, $p\restriction\gd\in P_\gd$. Moreover, $w(p)(M)$ is equal to $\trash$ for all nodes $M\in m(p)$ except
possibly some transitive nodes $M=V_\gd$ such that $f(\gd)$ is a $P_\gd$-name, $P_\gd\Vdash f(\gd)$ is a $\Phi$-proper forcing, $w(p)(V_\gd)$ is a $P_\gd$-name for an element 
of $f(\gd)$, and for 
every countable node $N\in m(p)$ such that $\{ P_\gd, f(\gd)\}\in N$,  $p\restriction\gd\Vdash_{P_\gd}$ the condition $w(p)(V_\gd)$ is $\Phi$-master for the model $N[\dot G_\gd]$.

Note that due to the closure of $m(p)$ under intersections and to the fact that the set $I$ consists of inaccessible cardinals, it is sufficient to verify the last condition for all countable nodes $N$
which are between $V_\gd$ and the next transitive node on $m(p)$.

The ordering is defined by $q\leq_\gb p$ if $m(q)\leq m(p)$ and for every transitive node $V_\gd\in m(p)$, if $w(p)(V_\gd)=\trash$ then $w(q)(V_\gd)=\trash$, and if $w(q)(V_\gd)\neq\trash$ then $q\restriction\gd\leq_\gb p\restriction\gd$ and $q\restriction\gd\Vdash m(q)(V_\gd)\leq m(p)(V_\gd)$.

\begin{claim}
$\leq_\gb$ is a transitive relation.
\end{claim}

\begin{proof}
This is an elementary argument by transfinite induction on $\gb$.
\end{proof}

Suppose that $\gd,\gb$ are elements of $I$ such that $\gd\in\gb$. Define $p_\gd^0$ to be the condition in $P_\gb$ which is $\langle \{V_\gd\}, \{\langle V_\gd, \trash\rangle\}\rangle$. In the event that $f(\gd)$ happens to be a $P_\gd$-name and $P_\gd\Vdash f(\gd)$ is $\Phi$-proper,
then also define $p_\gd^1$ to be the condition in $P_\gb$ which is $\langle \{V_\gd\}, \{\langle V_\gd, 1_{f(\gd)}\rangle\}\rangle$.

\begin{claim}
Suppose that $\gd, \gb\in I$ are ordinals such that $\gd\in\gb$, and $p\in P_\gb$ is a condition below $p_\gd^0$ or $p_\gd^1$. Suppose that $q\in P_\gd$ and $q\leq_\gd p\restriction\gd$. Then $r=\langle m_q\cup m_p, w(q)\cup (w(p)\setminus V_\gd)\rangle$ is a condition in
$P_\gb$ and $r\leq_\gb p$.
\end{claim}

\begin{proof}
This is another elementary argument by transfinite induction on $\gb$.
\end{proof}

If $\gd\in I$ is an ordinal less than $\kappa$ such that $f(\gd)$ is an $P_\gd$-name for an $\Phi$-proper forcing, we will write $P_{\gd+1}$ for the two-step iteration $P_\gd*f(\gd)$. For an ordinal $\gb>\gd$ and a condition $p\in P_\gb$ such that
$V_\gd\in m(p)$ and $w(p)(V_\gd)\neq\trash$, we will write $p\restriction\gd+1$ for the condition in $P_{\gd+1}$ which is the pair $\langle p\restriction\gd, w(p)(V_\gd)\rangle$. The following claim is now easy to show:

\begin{claim}
\label{ittclaim}
Let $\gd, \gb$ be ordinals in $I$ such that $\gd\in\gb$.

\begin{enumerate}
\item The conditions $p_\gd^0$ and $p_\gd^1$ both force the filter $G_\gb\cap P_\gd$ to be $P_\gd$-generic;
\item if $\gd$ is such that $f(\gd)$ is an $P_\gd$-name for a $\Phi$-proper forcing, then $p_\gd^1\Vdash G_\gb\cap P_{\gd+1}$ is $P_{\gd+1}$-generic.
\end{enumerate}
\end{claim}

Let $p,q\in P_\gb$ be conditions. Say that $p, q$ are in $\Delta$\emph{-position} if there is a countable node $M\in m(p)$ such that the model $M$ contains $q$ as well as $P_\gg$ and $f(\gg)$ for all $V_\gg\in m(q)$ as elements, and
writing $V_\gd$ for the largest transitive node on $m(p)$ below $M$, it is the case that $V_\gd\in m(q)$, $q\restriction\gd+1$ is compatible with $p\restriction\gd+1$ and all countable nodes of $m(p)$ between $V_\gd$ and $M$ belong to $m(q)$.
If there are no transitive nodes in $m(p)$ below $M$, then we require just that all countable nodes of $m(p)$ below $M$ belong to $m(q)$.

\begin{claim}
\label{deltaclaim}
If $p, q$ are in $\Delta$-position then they are compatible in $P_\gb$.
\end{claim}

\begin{proof}
Let $M\in m(p)$ be the model witnessing the $\Delta$-position of $p,q$.
Let us treat the case that there is a largest transitive node on $m(p)$ below $M$, denote it by $V_\gd$, and assume that $w(p)(V_\gd)\neq\trash$. Since $p\restriction\gd+1$ and $q\restriction\gd+1$ are compatible, there is a condition $r\in P_\gd$
below both $p\restriction\gd+1$ and $q\restriction \gd+1$, and a $P_\gd$-name $\tau$ such that $r\Vdash_\gd\tau\leq w(p)(V_\gd), w(q)(V_\gd)$ in the poset $f(\gd)$. To construct the lower bound $s$ of $p,q$, we must define $m(s)$ and $w(s)$.

Let $m(s)=m(r)\cup m(q)\cup m(p)\cup n$, where $n$ is the set of all intersections of the form $N\cap V_\gg$, where $V_\gg$ is a transitive node on $m(q)$ and $N=M$ or else $M$ is one of the countable nodes on $m(p)$ above $M$ such
that there is no transitive node between $M$ and $N$. First, we must verify that $m$ is a condition in $Q_\gd$. This is a mechanical checking of the clauses of the definition of $Q_\gd$; we will outline only the nontrivial points in the argument. For the closure of $m(s)$ under intersection, the only nontrivial case is to check is that if $N_0\in m(q)\setminus V_\gd$ and $N_1\in m(p)\setminus V_\gd$ are countable nodes, then $N_0\cap N_1\in m(r)$. To see this, note that $N_0\in M$ and so $N_0\subset M$, $N_0\cap N_1=N_0\cap (N_1\cap M)$, the countable node $N_1\cap M$ is in $m(p)$ by the closure of $m(p)$ under intersections, and there are two cases.

\textbf{Case 1.} Either $N_1\cap M\in V_\gd$. Then $N_0\cap N_1= (N_0\cap V_\gd)\cap (N_1\cap M\cap V_\gd)$. Now, the model $N_0\cap V_\gd\in m(q)$ by the closure of $m(q)$ under intersections, $N_1\cap M\cap V_\gd\in m(p)$ by the closure of $m(p)$ under intersection, so both of them are in $m(r)$ since $r\leq p\restriction\gd, q\restriction\gd$, and so $N_0\cap N_1\in m(r)$ by the closure of $m(r)$ under intersections.

\textbf{Case 2.} Or $N_1\cap M\notin V_\gd$. Then $N_1\cap M$ must be one of the models on $m(p)$ between $V_\gd$ and $M$, or it may be equal to $M$ itself. In the former case $N_1\cap M\in m(q)$ as $p,q$ are in $\Delta$-position, and so $N_0\cap N_1=N_0\cap N_1\cap M\in m(q)$ by the closure of $m(q)$ under intersections. In the latter case, $N_0\cap N_1=N_0\cap M=N_0\in m(q)$, and we are finished.

To verify that $m(s)$ forms an $\in$-chain, first observe that $m(r)\cup m(q)\cup m(p)$ is a concatenation of three $\in$-chains ($m(r)$, $m(q)\setminus V_\gd$, and $m_p\setminus M$, since $m(r)\in V_\gd$ and $m(q)\in M$) and so an $\in$-chain. Now inspect the models in the set $n$. Let $V_\gg$ be a transitive node in $m(q)$ and $K$ its predecessor
in $m(q)$. Then $K, V_\gg\in M$. Since the countable nodes between $M$ and the next transitive node in $m(p)$ above $M$ form an $\in$-chain and all contain $M$ as an element, they also contain $K, V_\gg$ and so their intersections with $V_\gg$ form an $\in$-chain whose nodes all contain $K$ and are contained in $V_\gg$. This immediately implies that $m(s)$ forms an $\in$-chain.

The definition of $w(s)$ breaks into  several cases, all of which except for one are trivial.

\noindent \textbf{Case 1.} For $V_\gg\in m(r)$, let $w(s)(V_\gg)=w(r)(V_\gg)$. 

\noindent \textbf{Case 2.} The value $w(s)(V_\gd)$ will be equal to $\tau$. This condition is forced to be $\Phi$-master for all the relevant models on $m(s)$ above $V_\gd$: these models come either from $m(p)$ or $m(q)$ or from intersections with transitive nodes, and $\tau$ is stronger than both $w(p)(V_\gd)$ and $w(q)(V_\gd)$.

\noindent \textbf{Case 3.} Now suppose that $\gg$ is such that $\gd\in \gg$ and $V_\gg\in m(q)$,
$f(\gg)$ is a $P_\gg$-name for an $\Phi$-proper poset, and $w(q)(V_\gg)\neq\trash$.
To define $w(s)(V_\gg)$, consider the set $n_\gg$ of all countable nodes in $m(s)$ below the next transitive node $V_{\gg^*}$ above $V_\gg$ (if such a node does not exist, just take all countable nodes
on $m(s)$) which contain $P_\gg, f(\gg)$. Observe that this
set is linearly ordered by $\in$, starts with (perhaps) some models on $m(q)$, after which comes $M\cap V_{\gg^*}$ and then (perhaps) some other models. By the definition of $P_\gb$, the condition $q\restriction\gg$ forces in $P_\gg$ that $w(q)(V_\gg)$ is a $\Phi$-master condition for $K[G_\gg]$ for all models $K\in n_\gg\cap m(q)$ and the poset $f(\gg)$. Moreover, $P_\gg, f(\gg)$ and $w(q)(V_\gg)$ all belong to the next model $M\cap V_{\gg^*}$
on the set $n_\gg$ beyond the models from $m(q)$. Therefore, using the definition of $\Phi$-master condition repeatedly, gradually strengthening the $P_\gg$-name for the condition $w(q)(V_\gg)$, it is possible to find a name $w(s)(V_\gg)$ forced by $q\restriction\gg$
to be $\Phi$-master for all models $N[G_\gg]$ where $N\in n_\gg$.

\noindent\textbf{Case 4.} To define $w(s)(V_\gg)$ for the transitive nodes $V_\gg$ on $m(s)$ above the model $M$, just let $w(s)(V_\gg)=w(p)(V_\gg)$.

It is not difficult to verify now that $s=\langle m(s), w(s)\rangle$ is a condition, and it is a lower bound of $p, q$ in $P_\gd$.
\end{proof}

For every countable elementary submodel $M\prec H_\theta$ containing $\kappa, f$ and an ordinal $\gb\in I$ let $p_M\in P_\gb$ be the unique condition with $m(p)=\{M\cap V_\gb\}$.

\begin{claim}
Let $\gb\in I$ be an ordinal. The poset $P_\gb$ is $\Phi$-proper, and for every countable elementary submodel $M\prec H_\theta$ containing the ordinal $\gb$ as well as $\kappa, f$, the condition $p_M$ is a $\Phi$-master condition for $M$ in the poset $P_\gb$.
\end{claim}

\begin{proof}
This is proved by induction on $\gb\in I$. Suppose that $\gb\in I$ is an ordinal below which the statement has been verified, and let $M\prec H_\theta$ be a countable elementary submodel containing $\gb$.

To verify that $p_M$ is a master condition $P_\gb$, suppose that $p\leq p_M$ is an arbitrary condition and $D\in M$ is an open dense subset of $P_\gb$; we must produce a condition $q\in D\cap M$ compatible with $p$.
Strengthening $p$ if necessary, we may assume that $p\in D$. For definiteness assume
that there are some transitive nodes in $m(p)$ below $M\cap V_\gb$, and let $V_\gd$ denote the largest one of them. For definiteness assume that $w(p)(V_\gd)\neq\trash$,
the other cases are simpler.

By the closure of $m(p)$ under intersections, $V_\gd\cap M\in m(p)$ holds, and by the induction hypothesis, $p\restriction\gd$ is a master condition for the model $M$ in the poset $P_\gd$.
By the definition of the poset $P_\gb$, $p\restriction\gd\Vdash w(p)(V_\gd)$ is a master condition for $M[\dot G_\gd]$. Therefore, $p\restriction\gd+1$ is a master condition for $M$ in the poset $P_{\gd+1}$.
Now, $p\restriction\gd+1$ forces in $P_{\gd+1}$ that there is a condition $q\in D$ such that $q\restriction\gd+1$ is in the generic filter $G_{\gd+1}$ and $m(q)$ contains all countable nodes of $m(p)$ between $V_\gd$
and $M$. This is clear since $q=p$ will work. Since $p\restriction\gd+1$ is $M$-master, it forces that there must be such a condition $q$ in the model $M$. In other words, there must be a condition $q\in M\cap D$ such that
$p\restriction\gd+1$ and $q\restriction\gd+1$ are compatible and $m(q)$ contains all countable nodes of $m(p)$ between $V_\gd$
and $M$. But then, $p,q$ are in $\Delta$-position and therefore compatible. The proof that $p_M$ is a master condition for the model $M$ is complete.

To verify that $p_M$ is a $\Phi$-master condition, suppose that $p\in P_\gb$ is a condition below $p_M$; we must find an element $g\in G\cap M$ and a set $F\in M$ on $RO(P_\gb)$ such that $\Phi(g, F, RO(P_\gb))$ holds and such that for every condition $q\in M\cap RO(P)$, if $q\geq p$ then $q\in F$. 
For definiteness assume that there are some transitive nodes in $m(p)$ below $M\cap V_\gb$, and let $V_\gd$ denote the largest one of them. For definiteness, also assume that $w(p)(V_\gd)\neq\trash$,
the other cases are simpler.

Let $\bar p=\langle m(\bar p), w(\bar p)\rangle$ be the condition defined in the following way: $m(\bar p)$ contains $V_\gd$, all the countable nodes of $m(p)$ between $V_\gd$ and $M$, and the intersections of these nodes with $V_\gd$. The map $w(\bar p)$ returns only one nontrivial value, at $V_\gd$, where it indicates the sum of all conditions in $f(\gd)$ which are $\Phi$-master for all models on $m(\bar p)$ containing $P_\gd$ and $f(\gd)$. It is clear that $\bar p\in M$ is a condition weaker than $p$. By the restriction clause of regularity, it will be enough to find the requested set $F$ in $RO(P_\gb\restriction\bar p)$.

By Claim~\ref{ittclaim}, the algebra $A=RO(P_{\gd+1}\restriction \bar p)$ can be naturally viewed as a complete subalgebra of $B=RO(P_\gb\restriction\bar p)$. By the induction hypothesis applied at $\gd$ and the two-step iteration Claim~\ref{itclaim},
the condition $p\restriction\gd+1$ is $\Phi$-master for $M$ and $P_{\gd+1}$. Thus, there are a set $F_0\subset A$ and an element $g\in G$ in the model $M$ such that $\Phi(g, F_0, A)$ holds and
$F_0$ contains all elements of $A\cap M$ weaker than $p\restriction\delta+1$. We will show that the set $F$, obtained as the upwards closure of
$F_0$ in the algebra $B$, has the requested properties.

Certainly $\Phi(g, F, B)$ holds by the subalgebra and closure clauses of regularity, and $F\in M$. 
We must verify that if $b\in B\cap M$ is weaker
than $p$ then $b\in F$. To this end, consider the lower projection function $\proj\colon B\to A$ defined by $\proj(b)=\sum\{a\in A\colon a\leq b\}\leq b$. We claim that if $b\in B\cap M$ is weaker than $p$, then $\proj(b)\in A\cap M$ is weaker than $p\restriction\gd+1$. This will complete the proof as then $\proj(b)\leq b$ must be an element of $F_0$ and so $b\in F$.

Suppose for contradiction that $p\restriction\gd+1$ is not stronger than $\proj(b)$. Then $p\restriction\gd+1$ must be compatible in $B$ with $1-b$ by the definition of projection. Since $p\restriction\gd+1$ is a master condition for $M$ by the first part of the proof of the claim, this means that there must be a condition $q\leq \bar p$ in the poset $P_\gb$ and the model $M$ such that $q\restriction\gd+1$ is compatible with $p\restriction\gd+1$, and $q$ is below $1-b$.
But then, $p$ and $q$ are in $\Delta$-position, as $m(q)$ contains all nodes on $m(\bar p)$ and so all countable nodes on $m(p)$ between $V_\gd$ and $M$. The conditions $p,q$ are therefore compatible by Claim~\ref{deltaclaim}. Their common lower bound will be below both
$p$ and $1-b$, contradicting the assumption that $p\leq b$.
\end{proof}

Now, suppose that $\kappa$ is a supercompact cardinal and $f\colon\kappa\to V_\kappa$ is the Laver prediction function. Let $P=P_\kappa$ be the $\Phi$-proper forcing obtained from the function $f$ using the scheme above. A routine argument now shows that $P$ forces $\Phi$-PFA to hold.
\end{proof}

The iteration theorem allows us to finally prove some interesting consistency results.

\begin{theorem}
YPFA implies

\begin{enumerate}
\item PID;
\item there are only five cofinal types of directed posets of size $\aleph_1$;
\item all Aronszajn trees are special;
\item $\mathfrak{c}=\aleph_2$.
\end{enumerate}

\noindent YPFA does not imply

\begin{enumerate}
\item[(5)] OCA;
\item[(6)] c.c.c.\ is productive.
\end{enumerate}
\end{theorem}

\noindent Items (4) and (6) are due to Todorcevic.

\begin{proof}
YPFA implies PID since the PID posets are Y-proper by Theorem~\ref{pidtheorem}. YPFA implies the five cofinal types statement since the posets used for it  are ideal-based \cite{z:keeping} 
and the ideal-based posets are Y-proper by Theorem~\ref{idealbasedtheorem}. Todorcevic remarked that the classification of cofinal types of size $\aleph_1$ is a consequence of the conjunction of PID and $\mathfrak{p}>\aleph_1$, which are both consequences of YPFA.

The proof of $\mathfrak{c}=\aleph_2$ follows the PFA argument with a small change. PID implies that $\mathfrak{b}\leq\aleph_2$ \cite{todorcevic:bsl}, and so YPFA implies $\mathfrak{b}=\aleph_2$. Similarly to the oldest proof that PFA implies $\mathfrak{c}=\aleph_2$, the argument is concluded by showing that $\mathfrak{c}=\mathfrak{b}$. This is quite involved and we only outline the main points. Fix a modulo finite increasing, unbounded sequence $\vec y=\langle y_\ga\colon\ga\in\gw_2\rangle$ of functions in the Baire space. For every $x\in\cantor$, consider a poset $P_x$ which is the iteration $P_x^0*\dot P_x^1$. The first step of the iteration is the $\in$-collapse of $\aleph_2$ to $\aleph_1$; the main point is that it is Y-proper and preserves unbounded sequences. The second step of the iteration uses $\vec y$ and the fact that it remains unbounded to code the point $x\in\cantor$ into a closed unbounded subset of $\gw_2^V$ via a c.c.c.\ poset; the main point is that it is in fact even Y-c.c. Thus, the iteration $P_x$ is Y-proper.
An application of YPFA to the poset $P_x$ yields coding of  the point $x$ by an ordinal $<\gw_2$, proving that $\mathfrak{c}=\aleph_2$. Details (except for showing that $\dot P_x^1$ is forced to be Y-c.c.) can be found in \cite[Theorem 3.16]{bekkali:set}.

For (5) and (6), Todorcevic supplied an argument using entangled linear orders. Suppose that there is a supercompact cardinal and
the continuum hypothesis holds. Then, there is a set $A$ of reals of size $\aleph_1$ which forms an entangled linear ordering \cite[Theorem 1]{todorcevic:entangled}. Note that entangledness is a statement about nonexistence of uncountable anticliques in a certain graph on $A^n$ and the graph in the case that $A$ is a set of reals is open. Force YPFA with a Y-proper ordering.  In the resulting model, $A$ is still entangled, thus OCA fails, and also by \cite[Theorem 6]{todorcevic:entangled} c.c.c.\ is not productive.
\end{proof}

\bibliographystyle{plain}
\bibliography{red-bib}

\begin{thebibliography}{10}

\bibitem{bpt:todorcevic}
Bohuslav Balcar, Tom{\' a}{\v s} Paz{\' a}k, and Egbert Th{\" u}mmel.
\newblock {On Todorcevic orderings}.
\newblock {\em Fundamenta Mathematicae}, 228:173--192, 2015.

\bibitem{baumgartner:thesis}
James~Earl Baumgartner.
\newblock {\em R{esults} {and} {independence} {proofs} {in} {combinatorial}
  {set} {theory}}.
\newblock ProQuest LLC, Ann Arbor, MI, 1970.
\newblock Thesis (Ph.D.)--University of California, Berkeley.

\bibitem{bekkali:set}
Mohamed Bekkali.
\newblock {\em Topics in set theory}, volume 1476 of {\em Lecture Notes in
  Mathematics}.
\newblock Springer-Verlag, 1991.

\bibitem{jech:newset}
Thomas Jech.
\newblock {\em Set theory}.
\newblock Springer Monographs in Mathematics. Springer-Verlag, Berlin, 2003.
\newblock The third millennium edition, revised and expanded.

\bibitem{mitchell:addclub}
William~J. Mitchell.
\newblock Adding closed unbounded subsets of {$\omega_2$} with finite forcing.
\newblock {\em Notre Dame J. Formal Logic}, 46(3):357--371, 2005.

\bibitem{neeman:twotypes}
Itay Neeman.
\newblock Forcing with sequences of models of two types.
\newblock {\em Notre Dame J. Form. Log.}, 55(2):265--298, 2014.

\bibitem{solecki:ideals}
S{\l}awomir Solecki.
\newblock Analytic ideals and their applications.
\newblock {\em Ann. Pure Appl. Logic}, 99(1-3):51--72, 1999.

\bibitem{steprans:strong-Q}
Juris Stepr{\=a}ns.
\newblock Strong {$Q$}-sequences and variations on {M}artin's axiom.
\newblock {\em Canad. J. Math.}, 37(4):730--746, 1985.

\bibitem{steprans:cellularity}
Juris Stepr{\=a}ns and Stephen Watson.
\newblock Cellularity of first countable spaces.
\newblock {\em Topology Appl.}, 28(2):141--145, 1988.
\newblock Special issue on set-theoretic topology.

\bibitem{thuemmel:horntarski}
Egbert Th{\"u}mmel.
\newblock The problem of {H}orn and {T}arski.
\newblock {\em Proc. Amer. Math. Soc.}, 142(6):1997--2000, 2014.

\bibitem{todorcevic:entangled}
Stevo Todorcevic.
\newblock {Remarks on chain conditions in products}.
\newblock {\em Compositio Mathematica}, 55:295--302, 1983.

\bibitem{todorcevic:cellularity}
Stevo Todorcevic.
\newblock {Remarks on cellularity in products}.
\newblock {\em Compositio Mathematica}, 57:357--372, 1986.

\bibitem{todorcevic:partitions}
Stevo Todor{\v{c}}evi{\'c}.
\newblock {\em Partition problems in topology}, volume~84 of {\em Contemporary
  Mathematics}.
\newblock American Mathematical Society, Providence, RI, 1989.

\bibitem{todorcevic:forcing}
Stevo Todor{\v{c}}evi{\'c}.
\newblock Two examples of {B}orel partially ordered sets with the countable
  chain condition.
\newblock {\em Proc. Amer. Math. Soc.}, 112(4):1125--1128, 1991.

\bibitem{todorcevic:pid}
Stevo Todor{\v{c}}evi{\'c}.
\newblock A dichotomy for {P}-ideals of countable sets.
\newblock {\em Fund. Math.}, 166(3):251--267, 2000.

\bibitem{todorcevic:bsl}
Stevo Todorcevic.
\newblock {Combinatorial dichotomies in set theory}.
\newblock {\em Bulletin of Symbolic Logic}, 17:1--72, 2011.

\bibitem{ulam:matrix}
Stanis{\l}aw Ulam.
\newblock {Zur {M}asstheorie der allgemeinen {M}engenlehre}.
\newblock {\em Fundamenta Mathematicae}, 16:140--150, 1930.

\bibitem{yorioka:random}
Teruyuki Yorioka.
\newblock Todorcevic orderings as examples of ccc forcings without adding
  random reals.
\newblock {\em Comment. Math. Univ. Carolin.}, 56(1):125--132, 2015.

\bibitem{z:keeping}
Jind{\v{r}}ich Zapletal.
\newblock Keeping additivity of the null ideal small.
\newblock {\em Proc. Amer. Math. Soc.}, 125(8):2443--2451, 1997.

\end{thebibliography}

\end{document}